\documentclass[10pt,reqno]{amsart}

\usepackage{amsmath}
\usepackage{amsfonts}
\usepackage{amssymb}
\usepackage{amsthm}
\usepackage{stmaryrd}

\usepackage{bm}

\usepackage{mathrsfs}
\usepackage{latexsym}
\usepackage{verbatim} 
\numberwithin{equation}{section}
\usepackage{enumitem}

\usepackage{graphicx}
\usepackage{subfigure}
\usepackage{epstopdf}

\usepackage{empheq}

\usepackage{ifthen} 

\provideboolean{shownotes} 
\setboolean{shownotes}{true} 
\newcommand{\margnote}[1]{
\ifthenelse{\boolean{shownotes}}%
{\marginpar{\raggedright\tiny\texttt{#1}}}%
{}%
}
\newcommand{\hole}[1]{
\ifthenelse{\boolean{shownotes}}%
{\begin{center} \fbox{ \rule {.25cm}{0cm}
\rule[-.1cm]{0cm}{.4cm} \parbox{.85\textwidth}{\begin{center}
\texttt{#1}\end{center}} \rule {.25cm}{0cm}}\end{center}}
{}
}

\theoremstyle{plain}

\newtheorem{lemma}{Lemma}[section]

\newtheorem{theo}[lemma]{Theorem}

\newtheorem{corollary}[lemma]{Corollary}

\theoremstyle{definition}

\newtheorem{remark}[lemma]{Remark}
\newtheorem{definition}[lemma]{Definition}

\theoremstyle{remark}

\usepackage{cite} 

\usepackage[colorlinks=true,urlcolor=blue,
citecolor=red,linkcolor=blue,linktocpage,pdfpagelabels,
bookmarksnumbered,bookmarksopen]{hyperref}

\usepackage{orcidlink}

\usepackage{cleveref}
\usepackage[pagewise,mathlines]{lineno}

\newcommand{\R}{{\mathbb R}}
\newcommand{\Ic}{{\mathbb I}}
\newcommand{\M}{\mathcal{M}}
\newcommand{\Z}{{\mathbb Z}}
\newcommand{\C}{{\mathbb C}}

\newcommand{\N}{{\mathbb N}}

\newcommand{\Id}{\mathbf{I}_d}

\newcommand{\Ido}{\mathbf{I}_2}

\newcommand{\cEw}{\mathscr{E}}

\newcommand{\cU}{{\mathcal{U}}}

\newcommand{\cJ}{{\mathcal{J}}}

 \newcommand{\llb}{\llbracket}
\newcommand{\rrb}{\rrbracket}
\newcommand{\eu}{\hat{\bm{e}}_1}
\newcommand{\ed}{\hat{\bm{e}}_2}

\newcommand{\Fp}{\underline{\bm{F}}^{+}}
\newcommand{\Fm}{\underline{\bm{F}}^{-}}
\newcommand{\bB}{\boldsymbol{B}}
\newcommand{\bN}{\boldsymbol{N}}
\newcommand{\bF}{\boldsymbol{F}}
\newcommand{\bbf}{\boldsymbol{f}}
\newcommand{\bn}{\boldsymbol{n}}
\newcommand{\bFs}{\boldsymbol{F}^{\pm}}
\newcommand{\bFu}{\underline{\boldsymbol{F}}^{\pm}}

\newcommand{\Cof}{\mathrm{Cof}\,}
\newcommand{\sgn}{\mathrm{sgn}\,}

\newcommand{\ii}{\mathrm{i}}

\renewcommand{\Re}{\mbox{\rm Re}\,}
\renewcommand{\Im}{\mbox{\rm Im}\,}

\newcommand{\Tr}{\mbox{\rm tr}\,}
\newcommand{\Div}{\mbox{\rm div}}

\newcommand{\ShowColoredChanges}{true} 
\ifthenelse{\isundefined{\ShowColoredChanges}}
  {

  }
  {

  }

\begin{document}

\title[Stability of phase boundaries for Hadamard hyperelastic materials]{Dynamical stability of planar phase boundaries for hyperelastic materials of Hadamard type}

\author[H. Freist\"{u}hler]{Heinrich Freist\"{u}hler \orcidlink{0000-0002-0741-886X}}

\address{{\rm (H. Freist\"{u}hler)} Department of Mathematics\\University of Konstanz\\78457 Konstanz (Germany)} 

\email{heinrich.freistuehler@uni-konstanz.de}

\author[L. Morales]{Lauro Morales \orcidlink{0000-0003-3294-755X}}

\address{{\rm (L. Morales)} Departamento de Matem\'aticas\\Universidad Aut\'onoma Metropolitana\\Unidad Iztapalapa\\Av. San Rafael Atlixco 186, Col. Vicentina, C.P. 09340\\Cd. de M\'{e}xico (Mexico)}

\email{lauro\_mm@xanum.uam.mx}

\author[R. G. Plaza]{Ram\'on G. Plaza \orcidlink{0000-0001-8293-0006}}

\address{{\rm (R. G. Plaza)} Departamento de Matem\'aticas y Mec\'anica\\Instituto de Investigaciones en Matem\'aticas Aplicadas y en Sistemas\\Universidad Nacional Aut\'onoma de M\'exico\\Circuito Escolar s/n, Ciudad Universitaria\\C.P. 04510 Cd. de M\'{e}xico (Mexico)}

\email{plaza@aries.iimas.unam.mx}

\author[F. Vallejo]{Fabio Vallejo \orcidlink{0009-0004-2580-7414}}

\address{{\rm (F. Vallejo)} Departamento de Matem\'aticas y Mec\'anica\\Instituto de Investigaciones en Matem\'aticas Aplicadas y en Sistemas\\Universidad Nacional Aut\'onoma de M\'exico\\Circuito Escolar s/n, Ciudad Universitaria\\C.P. 04510 Cd. de M\'{e}xico (Mexico)}

\email{fabioval@ciencias.unam.mx}

\begin{abstract}
The dynamical stability of laminates or planar phase boundaries for hyperelastic materials of Hadamard type in two space dimensions is studied. For that purpose, the stability function, known as the Lopatinski\u{\i} determinant, is computed for states of deformation at both sides of the planar interface that account for the generalized Legendre-Hadamard conditions derived by Grabovsky and Truskinovsky (J. Elast. 123 (2016), 225--243). The sufficient conditions for the dynamical stability of such configurations are described in terms of the physical parameters of the model, such as the shear modulus, and computed under kinetic conditions across the interface of both Maxwell (conservation of energy) or Abeyaratne and Knowles (dissipation of energy) types.
\end{abstract}

\keywords{Elastic phase boundaries, compressible Hadamard materials, multidimensional stability, Lopatinski\u{\i} determinant.}

\subjclass[2020]{74A50, 35Q74, 35B35, 74H55, 74N20.}

\maketitle
\setcounter{tocdepth}{1}



\section{Introduction}
\label{secintro}

The analysis of phase boundaries of shear bands in elastic materials is one of the fundamental problems in elasticity theory. There is a vast literature on the description, variational stability and emergence of phase transitions, interpreted as static configurations delimiting local minimizers of an energy functional that defines the material under consideration. For an abridged list of references, the reader is referred to \cite{GrTr16,GrTr19,Dac08,Siva88,BJ87,BJ92,Ba4}. The present work aims to contribute to the stability theory of such configurations from a \emph{dynamical} viewpoint and considers planar two-phase configurations for hyperelastic \emph{Hadamard materials}.

The term Hadamard material was coined by John \cite{Joh66} (based on an early description by Hadamard \cite{Hada1903}) to account for a large class of elastic media where the stored energy density function is of the form
\begin{equation}
\label{Hadamardenergy}
W(\bF) = \frac{\mu}{2} \Tr (\bF^\top \bF) + h(\det \bF),
\end{equation}
where $\bF \in \R^{d \times d}$, with $\det \bF >0$, denotes the deformation gradient, $\Tr(\cdot)$ denotes the trace function, $\mu > 0$ is the shear modulus and the scalar function $h = h(J)$, defined for $J = \det \bF \in (0,\infty)$, is known as the volumetric density function. In Eq. \eqref{Hadamardenergy} the trace term is the isochoric (or isovolumetric) part of the energy, quantifying energy changes at constant volume, whereas $h$ accounts for energy changes due to changes in volume. The simplest interpretation of an elastic Hadamard material is as a compressible extension of a neo-Hookean incompressible solid and energies of the form \eqref{Hadamardenergy} usually model martensitic solid phase transitions. Following Grabovsky and Truskinovsky \cite{GrTr16}, in this paper we specialize our analysis to the two-dimensional case, for which $d = 2$.

The concept of stability addressed in this paper is of dynamical nature. Looking at static weak solutions to the equations of elastodynamics, the main question that we pose is the following: given a multidimensional smooth perturbation (or a small wave impinging on the interface) of the static phase boundary, is there a local solution to the equations of elastodynamics with the same structure? In other words, will solutions to the dynamical system of equations, whose initial data are near to the planar interface, be close and similar to (or far and qualitatively different from) it? The problem was addressed for the first time by Freist\"uhler and Plaza \cite{FrP1,FrP3}, who set up the theoretical basis for a rigorous stability theory of such configurations. The main idea of the authors in \cite{FrP1} was to consider the static configurations as (non-classical) \emph{stationary undercompressive shock fronts} \cite{Fr3,Da4e,LeF02} for the equations of elastodynamics, which are of conservative nature. Therefore, Majda \cite{M1,M2} and M\'etivier's \cite{Me1,Me2,Me5} theory of stability of planar shock fronts can be adapted to the undercompressive case (cf. Freist\"uhler \cite{Fr3}) and, in turn, applied to the case of elastic phase boundaries. The article \cite{FrP1} contains the theoretical framework and the reduced stability conditions to perform such analysis.  As a result from their work, it is now known that the nonlinear stability of planar elastic phase configurations depends upon the algebraic \emph{Lopatinski\u{\i} conditions} for linear hyperbolic initial boundary value problems \cite{BS,Kre70,Lopa70,Lopa56} (also known as the \emph{hyperbolic Kreiss-Majda-Lopatinski\u{\i} conditions}). In a later contribution \cite{PL2}, these conditions were verified numerically in the case of a two-well hyperelastic energy density in three dimensions modeling martensite twins.

In this paper we establish for the first time the dynamical stability conditions for two phase planar elastic boundaries in the case of two-dimensional Hadamard materials. The main differences with the study in \cite{PL2} consist of, on one hand, the fundamental assumption that the deformation gradients defining the static configuration (or the energetic wells on both sides of the interface) \emph{are not local minimizers of the stored energy density function}. They are admissible, however, in the sense that classical phase coexistence conditions are satisfied, as well as the generalized Legendre-Hadamard conditions derived in \cite{GrTr16}, which account for stability in the variational framework. On the other hand, our result are theoretical as we rigorously establish the conditions for stability via a direct calculation of the Lopatinski\u{\i} stability function (in contrast, the results of \cite{PL2} are numerical). Up to our knowledge, ours is the first analytic proof of the dynamical stability of elastic phase boundaries. 

It is well-known that for shocks of undercompressive nature, a certain kinetic relation (or kinetic rule) is needed to determine the dynamics of the shock \cite{FrP1,Fr3}. This kinetic rule induces a jump relation of Rankine-Hugoniot type across the interface. In this paper we consider kinetic rules as those proposed by Abeyaratne and Knowles \cite{AK90,AK91}, prescribed as constitutive relations which, in the most basic scenario, relate the normal velocity of the discontinuity with the driving force across the interface. The simplest of those kinetic relations is known as the Maxwell kinetic rule, which is tantamount to conservation of energy across the interface. Our results apply to phase transitions under both the Maxwell rule and kinetic relations of Abeyaratne and Knowles type expressing energy dissipation at the interface.

\subsection*{Plan of the paper}

This paper is laid out as follows. In Section \ref{secmodel} we describe the equations of elastodynamics as a first order system of conservation laws, for which elastic phase boundaries arise as weak solutions. In Section \ref{secelastipb} we describe the static configurations that model laminates or two-phase interfaces and the kinetic relation that determines their dynamics. In particular, in Section \ref{secstabtheory} we briefly recall the stability theory from \cite{FrP1}. Section \ref{sechada} is devoted to the description of these configurations in the case of two-dimensional Hadamard materials, following the analysis of Grabovsky and Truskinovsky \cite{GrTr16}. Sections \ref{RedLop1} and \ref{seclopa} are devoted to computing the Lopatinski\u{\i} stability function and to stating our main results; see Theorems \ref{MaxSta} (under conservation of energy or kinetic relation of Maxwell type) and \ref{AbeySta} (under kinetic rules of Abeyaratne and Knowles type which dissipate energy) below. Section \ref{secdisc} contains a brief discussion of the results and some final remarks. For convenience of the reader, we have included in Appendix \ref{GrabTro} a description of the generalized Legendre-Hadamard conditions derived by Grabovsky and Truskinovsky \cite{GrTr16}, as well as the calculations of such conditions for Hadamard materials in Appendix \ref{L-Hgt}.

\subsection*{On notation}

The unit imaginary number is denoted by $\mathrm{i} \in \C$, $\ii^2 = -1$, and $i,j \in \Z$ indicate integer indices. For a complex number $\lambda$, we denote complex conjugation by ${\lambda}^*$ and its real and imaginary parts by $\Re \lambda$ and $\Im \lambda$, respectively.  Real matrices are denoted by capital roman font letters (e.g. $\bm{A} \in \R^{d \times d}$), except for the first Piola-Kirchhoff stress tensor, denoted by $\bm{\sigma}$. Complex matrix fields in the space of frequencies will be denoted with calligraphic letters (e.g., $\mathcal{A} \in \C^{n \times n}$). Complex transposition of block matrices are indicated by the symbol ${}^*$ (e.g., $\mathcal{A}^*$), whereas simple transposition is denoted by the symbol $\bm{A}^\top$. $\{ \hat{\bm{e}}_j \}_{j=1}^d$ denotes the canonical basis of $\R^d$ and $\Id$ denotes the identity $d \times d$ matrix, for each $d \in \N$, $d \geq 2$. $\R_+^{d\times d}$ is the set of all real $d \times d$ matrices with positive determinant. In this paper, the elements of a real matrix $\bm{A}$ will be denoted as $A_{ij}$ and $\bm{A}_j$ will denote the $j$-th column vector. We denote the cofactor matrix of any real square matrix $\bm{A} \in \R^{d\times d}$ to be $(\Cof \bm{A})_{ij} = (-1)^{i+j} \det (\bm{A}'_{(i,j)})$, $1 \leq i,j \leq d$, where $\bm{A}'_{(i,j)}$ is the $(d-1) \times (d-1)$ matrix obtained by deleting the $i$-th row and the $j$-th column of $\bm{A}$. Hence,
\begin{equation}
\label{exprcof}
(\Cof \bm{A})^\top \bm{A} = \bm{A} (\Cof \bm{A})^\top = (\det \bm{A}) \Id.
\end{equation}For any $\bm{a}, \bm{b} \in \R^d$, $\bm{a} \otimes \bm{b} \in \R^{d \times d}$ denotes the standard tensor product matrix whose $(i,j)$-entry is $a_i b_j$, whereas $\bm{a}\cdot\bm{b}=\bm{a}^{\top}\bm{b}\in\R$ denotes the usual dot product in $\R^{d}$. For any (scalar or matrix valued) function $g$ of the state variables $\bm{u}$, the jump across a shock discontinuity will be denoted as $\llb g(\bm{u}) \rrb := g(\bm{u}^+) - g(\bm{u}^-)$, whereas its mean across the interface will be indicated by $\langle g \rangle := \tfrac{1}{2} (g(\bm{u}^+) + g(\bm{u}^-))$.

\section{Elastodynamics}
\label{secmodel}

\subsection{The equations of motion}
\label{model1}

In the theory of elastodynamics, any deformation of an elastic material sample or reference configuration $\Omega\subset \R^d$ (typically, $d=1,2$ or $3$) is described by a function or displacement field, $\bm{y}:\Omega\times [0,\infty)\to\R^d$, in such a way that a sample point located at $\bm{x}\in \Omega$ at $t=0$ evolves to the point $\bm{y}(\bm{x},t)$ at time $t>0$. For every deformation $\bm{y}:\Omega\times [0,\infty)\to\R^d$ regular enough, we define the deformation gradient field, $\bF:\Omega\times [0,\infty)\to \R^{d\times d}$, and  the local velocity vector, $\bm{v}:\Omega\times [0,\infty)\to \R^d$, as
\[
\bF(\bm{x},t)=\nabla_{x}  \, \bm{y}\in\R^{d\times d},\quad \text{and}\quad \bm{v}(\bm{x},t)=\bm{y}_t\in\R^d,
\]
respectively, or equivalently as
\begin{equation}
\label{Matder1}
F_{ij}=\dfrac{\partial y_i}{\partial x_j},\quad\text{and}\quad v_i =\dfrac{\partial y_i}{\partial t},
\end{equation}
for all $i,j=1,\cdots,d$. We say that the elastic sample is \textit{hyperelastic} if there exists a single stored energy density function $W:\R_+^{d\times d}\to [0,\infty)$, defined per unit volume in the reference configuration, from which the stress-strain relation, or constitutive relation, can be derived. This is given by the first Piola-Kirchhoff stress tensor, $\bm{\bm{\sigma}}=\bm{\bm{\sigma}}(\bF)$, which derives from $W$ as    
\[
\bm{\sigma} (\bF):=\frac{\partial W}{\partial \bF},
\]
or component-wise as 
\[
\sigma_{ij} (\bF):=\frac{\partial W}{\partial F_{ij}},\quad 1\leq i,\: j\leq d.
\]
Fundamental assumptions on the strain-energy function $W$ include the principle of frame-indifference, material symmetry (isotropy), and objectivity; see, for example, Ogden \cite{Og84}. 

In the absence of external forces, Newton's law readily implies the well-known non-thermal elastodynamics system 
\begin{equation}
\label{model0}
\begin{array}{r}
\bm{y}_{tt} - \Div_{x}  \, \bm{\sigma}(\bF) = 0,
\end{array}
\end{equation}
which can be reformulated as a first-order system of conservation laws in terms of the deformation gradient $\bF$ and the velocity $\bm{v}$ as follows, 
\begin{equation}
\label{model}
\begin{array}{r}
\bF_t - \nabla_{x}  \bm{v} = 0, \\
\bm{v}_t - \Div_{x}  \, \bm{\sigma}(\bF) = 0,
\end{array}
\end{equation}
or component-wise as 
\begin{equation}
\label{modelij}
\begin{aligned}
\partial_t F_{ij} - \partial_{x_j} v_i &= 0, & \quad i,j = 1, \ldots, d,\\
\partial_t v_i - \sum_{j=1}^d \partial_{x_j} \Big( \frac{\partial W}{\partial F_{ij}} \Big) &= 0, &\quad i = 1, \ldots, d,
\end{aligned}
\end{equation}
with the constraint 
\begin{equation}
\label{curlfree}
\hbox{curl}_x \, \bF = 0,
\end{equation}
which is known as the curl-free constraint and it is simply a short-cut for the compatibility equations 
\begin{equation}
\label{curlfreeij}
\partial_{x_k} F_{ij} = \partial_{x_j} F_{ik}, \qquad i,j,k = 1,\cdots, d.
\end{equation}
The latter are clearly a consequence of Eq. \eqref{Matder1}. Following the notation in \cite{FrP1}, we write the  $j$-th column of $\bF$ as $\bF_j$ and the  $j$-th column of $\bm{\sigma}(\bF)$ as ${\bm{\sigma}(\bF)}_j$. That is
\[
\bF_j :=\begin{pmatrix}
    F_{1j} \\ \vdots \\ F_{dj} 
\end{pmatrix}\in \R^{d}, \quad \mbox{and}\quad {\bm{\sigma}(\bF)}_j:=\begin{pmatrix}
 \frac{\partial W}{\partial F_{1j}} \\ \vdots \\ \frac{\partial W}{\partial F_{dj}}
\end{pmatrix}\in \R^{d},\quad j=1\cdots d.
\]
With a slight abuse of notation, if we consider $\bF$ as the column vector $\bF=(\bF_1^{\top},\cdots,\bF_d^{\top})^{\top}\in\R^{d^2}$ then system \eqref{model} can be recast as a system of $n = d^2 + d$ conservations laws of the form (cf. \cite{PlVa22,FrP1}),
\begin{equation}
\label{HSCL}
\bm{u}_t + \sum_{j=1}^d \bbf^j(\bm{u})_{x_j} = 0,
\end{equation}
where 
\begin{equation}
\label{ufus}
\bm{u} = \begin{pmatrix} \bF_1 \\ \vdots \\ \bF_d \\ \bm{v}
\end{pmatrix} \in \R^{d^2 + d}, \qquad \bbf^j(\bm{u}) = - \begin{pmatrix} 0 \\ \vdots \\ \bm{v} \\ \vdots \\ 0 \\ \bm{\sigma}(\bF)_j \end{pmatrix} \in \R^{d^2 + d},
\end{equation}
denote the conserved quantities $\bm{u}$ belonging to an admissible set $\mathcal{U} \subset \R^n$ of state variables, and the nonlinear fluxes $\bbf^j : \cU \to \R^n$, $\bbf^j \in C^2(\cU; \R^n)$, $1 \leq j \leq d$, respectively. In the expression for $\bbf^j $ in \eqref{ufus}, the column vector $\bm{v}$ is located in the $j$-th $d\times 1$ block from the top.
%

From physical considerations, it is usual to assume that $\det \bF>0$ (cf. Ciarlet \cite{Ci}). Thus, the open and connected set of admissible states is 
\[
\cU = \{ (\bF,\bm{v}) \in \R^{d \times d} \times \R^d \, : \, \det \bF > 0 \}. 
\]
Under this notation, the Jacobians 
$\bm{A}^j(\bm{u}) := D\bbf^j(\bm{u}) \in \R^{n \times n}$ are given by
\[
\bm{A}^j(\bm{u}) = - \begin{pmatrix} & & & 0 \\ & & & \vdots  \\ & 0 & & \Id \\ & & & \vdots  \\ & & & 0 \\ \bB_1^j(\bF) & \cdots & \bB_d^j(\bF) & 0  \end{pmatrix} \in \R^{(d^2 + d) \times (d^2 + d)},
\]
for all $j = 1, \ldots, d$, where the matrix fields $\bB_i^j$, for each $i,j=1,\cdots, d$, denote the second derivatives of the energy $W$ and they are defined by $\bB_i^j(\bF):=\partial\bm{\sigma}_j/\partial \bF_i$; more precisely, 
\begin{equation}
\label{defBij}
\bB_i^j(\bF) := \frac{\partial \bm{\sigma}_j}{\partial \bF_i} = \begin{pmatrix} W_{F_{1j}F_{1i}} & \cdots & W_{F_{1j}F_{di}} \\ \vdots & & \vdots \\ W_{F_{dj} F_{1i}} & \cdots & W_{F_{dj} F_{di}}\end{pmatrix} \in \R^{d \times d}
\end{equation}
(see \cite{FrP1} for details). By the previous relations, it follows that $(\bB_i^j)^{\top}=\bB_j^i$ for all $i,j=1,\cdots, d$.  

The hyperbolicity of system \eqref{HSCL} is a necessary condition for the well-posedness of the Cauchy problem \cite{Da4e}. For the system under consideration, the hyperbolicity depends on the stored energy density function $W$ and it is linked to the well known Legendre-Hadamard condition \cite{FrP1}. The latter is given in terms of the \textit{acoustic tensor},  defined as
\begin{equation}
\label{defacouten}
\bN(\bm{\xi},\bF) := \sum_{i,j=1}^d \xi_i \xi_j \bB_i^j(\bF) \in \R^{d \times d},
\end{equation} 
for all $\bF\in\R^{d\times d}$ and $\bm{\xi}\in\R^d$. The Legendre-Hadamard condition asserts that the acoustic tensor is positive definite for all non zero frequencies $\bm{\xi}\in\R^{d}$. More precisely, 
\begin{equation}
\label{LHcond0}
\bm{\eta}^{\top} \bN(\bm{\xi},\bF) \bm{\eta} > 0, \quad \text{for all } \; \bm{\xi}, \bm{\eta} \in \R^d \backslash \{0\}.
\end{equation}
This condition guarantees the convexity of $W$ along rank-one lines passing through $\bF$ and, in such a case, it is customary to say that $W$ is rank-one convex at $\bF$. This property is an essential tool for delimiting stability boundaries for weak local minima in the Calculus of Variations \cite{GiaHil96vI}. Another essential assumption for the analysis is the so called  \textit{constant multiplicity} assumption of M\'etivier \cite{Me4,FrP1}. Specifically, this property  holds for the stored energy density function $W$ at  $\bF$ if, for all nonzero frequencies $\bm{\xi}\in\R^d$, the eigenvalues of the acoustic tensor  $\bN = \bN(\bm{\xi},\bF)$ have multiplicities which remain constant with respect to $\bm{\xi}$ and $\bF$. 

\subsection{Planar laminates as weak solutions to the equations of elastodynamics} 

If we neglect boundary conditions ($\Omega=\R^d$), it follows that, for any fixed $\bF_0\in\R^{d\times d}$, there holds that constant tensor function  $(\bF,\bm{v})=(\bF_0,\bm{0})$ is an stationary weak solution to \eqref{model} with initial data $\bF(\bm{x},0)=\bF_0$. In terms of deformations, this constant state corresponds to the steady affine deformation
\[
\bm{y}(\bm{x},t) = \bF_0\bm{x}, \qquad \mbox{for all}\:\: \bm{x}\in\R^d.
\]
This observation allows to construct piecewise affine functions in $\R^d$. Indeed, if $\bF^+$ and $\bF^{-}$ are two constant matrices in $\R^{d\times d}$ such that there exist vectors $\bm{a},\bn\in\R^d$,  with $\bm{a}\neq0$ and $|\bn|=1$ satisfying 
\begin{equation}
\label{eq:rank-one}
\bF^{+}-\bF^{-}=\bm{a}\otimes\bn,
\end{equation}
then the steady deformation 
\begin{equation} 
\label{eq:twin}
\bm{y}(\bm{x},t):=\bm{y}_0 +\begin{cases} \bF^-\bm{x},  &  \bn\cdot \bm{x}< 0,\\ \bF^+\bm{x},  & \bn\cdot \bm{x} > 0,
\end{cases},
\qquad  \bm{x} \in \R^d,
\end{equation}
is  continuous with $\nabla_x \bm{y}\in \{\bF^-,\bF^+\}$ a.e. in $\R^d$. The surface of gradient discontinuity is given by $\{x\cdot \bn=0\}$\footnote{When $\bF^+$, $\bF^-$ are local minima of the stored energy density function $W$, corresponding to a phase transition between two martensitic phases, the steady deformation \eqref{eq:twin} is called a martensite twin \cite{PL2}.}. In terms of the variables $(\bF,\bm{v})$, \eqref{eq:twin} becomes
\begin{equation}
 \label{shock}
(\bF_*,\bm{v}_*)(\bm{x},t):= \begin{cases} (\bF^{-},0),  &  \bn\cdot \bm{x} < 0,\\ (\bF^+,0),  & \bn\cdot \bm{x} > 0,
\end{cases} \quad (\bm{x},t)\in\R^d\times [0,\infty),
\end{equation}
and it is a time-independent (static) weak solution to the system of conservation laws \eqref{model}-\eqref{curlfree}. The relation \eqref{eq:rank-one} is known as the \textit{Hadamard condition} and in this case it is customary to say that $\bF^-$ and $\bF^+$ are rank-one connected. 

\section{Elastic phase boundaries}
\label{secelastipb}

\subsection{Jump conditions}

Suppose the energy density function $W$ is rank-one convex at two constant states $\bF=\bFs\in\R_+^{d\times d}$. This implies that the system \eqref{model} is hyperbolic in open neighborhoods of $\bF=\bFs$, an essential prerequisite already for the stability of the individual phases (see \cite{FrP1}). A function of the form    
\begin{equation}
 \label{shock1}
(\bF,\bm{v})(\bm{x},t) = \begin{cases} (\bF^-,\bm{v}^-),  & \bm{x}\cdot \bn - st<0,\\ (\bF^+,\bm{v}^+),  & \bm{x} \cdot \bn - st>0,
\end{cases}
\end{equation}
represents a \textit{moving} phase boundary that travels in the constant direction $\bn\in\R^d$, $|\bn|=1$, with speed $s\in\R$. For short, \eqref{shock1} is denoted as 
\[
((\bF^-,\bm{v}^{-}),(\bF^+,\bm{v}^{+}),s,\bn\big)=(\bFs,\bm{v}^{\pm},s,\bn).
\]
The configuration \eqref{shock}, for example, corresponds to a \emph{static} phase boundary of the form $(\bFs,\bm{0},0,\bn)$.  If we assume that the phase boundary \eqref{shock1} is a weak solution to system \eqref{model} then the canonical Rankine-Hugoniot jump conditions associated with Eq. \eqref{model} must be satisfied. The latter have the form (cf. \cite{FrP1}),
\begin{equation}
\label{RHe1}
\begin{aligned}
-s \llb \bF \rrb - \llb \bm{v} \rrb\otimes \bn &= 0, \\
-s \llb \bm{v} \rrb - \llb \bm{\sigma}(\bF) \rrb\bn &= 0,\\
\llb \bF\rrb\times \bn &= 0,
\end{aligned}
\end{equation}
where $\llb g\rrb=g^+-g^-$ denotes the jump of any tensor or scalar field $g$.  The third jump condition above is associated with the curl-free constraint \eqref{curlfree}.
\begin{remark}
\label{crosrel}
In the three dimensional case ($d=3$), $\llb \bF\rrb\times \bn$ corresponds to the matrix whose $k$-th row is defined as the standard cross product in $\R^3$ between the $k$-th row of $\llb \bF\rrb$ with $\bn$. In dimension $d=2$, it is defined as
\[
\llb \bF\rrb\times \bn:=-\llb \bF\rrb\bn^{\perp},
\]
where $\bn^{\perp}:=(-n_2,n_1)^{\top}$ denotes the normal vector to $\bn=(n_1,n_2)$. This last definition is consistent with that of the three-dimensional case, as the above expression is recovered from $\llb \bF\rrb\times \bn$ in dimension $d=3$, by setting the third component of $\bn$ to zero.
\end{remark}

We are particularly interested in \textit{subsonic} phase boundaries, that is, moving phase boundaries whose traveling speed $s$ satisfies
\begin{equation}
\label{subSp1}
0\leq s^2<\min\{\text{eigenvalues of}\:\:\bN(\bn,\bFs)\}.
\end{equation}
Under the above assumptions and in the context of multidimensional stability of shock waves \cite{FrP1,Cou03}, the phase boundary $(\bFs,\bm{0},0,\bn)$ corresponds to an undercompressive shock \cite{Fr95}, as the sum of the number of outgoing characteristics at both  sides of the shock equals the dimension of the state space $n=d^2+d$ (see \cite{Fr3,FrP1}).

\subsection{Kinetic rules}
Given $(\bFs,\bm{v}^{\pm})$ and $\bn$, the speed $s$ of the associated subsonic phase boundary $(\bFs,\bm{v}^{\pm},s,\bn)$ is not uniquely determined by Eqs. \eqref{model} and \eqref{RHe1}. Thus, further jump conditions, relating the states on either side of the moving boundary with its space-time normal $(\bn,s)$, need to be added to circumvent the problem of non-uniqueness. These extra conditions are known as \textit{kinetic rules} and, in general, they have the following functional form  
\begin{equation}\label{King0}
g\big((\bF^-,\bm{v}^{-}),(\bF^+,\bm{v}^{+}),s,\bn\big)=0,
\end{equation}
where $g$ is, at least, a differentiable real valued function. In real elastic media, the precise circumstances under which the motion of phase boundaries is captured by a kinetic rule, as in \eqref{King0}, seem currently not clear from the literature. In this work we proceed as in \cite{PL2}, assuming such circumstances and considering a kinetic rule defined by $g$ of the form
\begin{equation}\label{King1}
g=\mathscr{F}+\tilde{g},
\end{equation}
where
\begin{equation}\label{King2}
\mathscr{F}=\llb W \rrb - \bn^{\top}\llb \bF\rrb^\top\langle \bm{\sigma} (\bF)\rangle\bn,
\end{equation}
is often called \textit{the driving traction} (or driving force) across the boundary, and $\tilde{g}$ is a differentiable function. The special case when
\begin{equation}\label{MaxRu}
\tilde{g}\equiv0,
\end{equation}
is known as the \emph{Maxwell (or Hugoniot) rule}. This kinetic rule corresponds to the conservation of energy across the interface \cite{AK90}. Based on the principle of thermodynamic irreversibility,  Abeyaratne and Knowles \cite{AK90,AK91} introduced a class of kinetic rules by requiring $\tilde{g}$ to meet the following conditions:
\begin{subequations}\label{AbeKR}
\begin{align}
&\tilde{g}\quad\text{is a differentiable function of its parameters,}\\
&\tilde{g}\big((\cdot,\cdot),(\cdot,\cdot),0,\cdot \big)=0\quad \text{identically,}\label{AbeKR02}\\
&\tilde{g}>0\quad\text{for}\quad s<0,\quad \tilde{g}<0\quad\text{for}\quad s>0,\\
&(D_s \tilde{g})\Big|_{\big((\bF^-,\bm{0}),(\bF^+,\bm{0}),0,\bn\big)}<0.
\end{align}
\end{subequations}

The aim of this paper is to study the dynamical stability, as solutions to \eqref{HSCL}, of a family of \textit{static} ($s=0$)  phase boundaries of the form \eqref{shock} under small dynamical perturbations. The analysis is carried out under the assumptions of either the Maxwell or the Abeyaratne-Knowles kinetic rules.  

The Rankine-Hugoniot conditions and the kinetic Maxwell rule  $\mathscr{\bF}=0$, are compatible with the classical Weierstrass-Erdmann conditions, which are widely used in the modeling of phase coexistence in the framework of nonlinear elasticity \cite{Erd1877}. Indeed,  observe that if the (static) phase boundary  $(\bF^{\pm},\bm{0},0,\bn)$ satisfies the Rankine-Hugoniot conditions \eqref{RHe1} then the second equation in \eqref{RHe1} simplifies to 
\begin{equation}\label{halfgra1}
\llb \bm{\sigma}(\bF) \rrb\bn = 0,
\end{equation}
because $s=0$. This is known as the first Weierstrass-Erdmann condition for two rank one connected states $\bF^{\pm}$. On the other hand, keeping in mind that for two vectors $\bm{v}$, $\bm{w}$ in $\R^d$, $\bm{v}\otimes\bm{w}:=\bm{v}\bm{w}^{\top}$ and $\Tr(\bm{v}\otimes\bm{w})=\bm{v}\cdot\bm{w}=\bm{v}^{\top}\bm{w}$, one verifies that the Maxwell rule, $\mathscr{\bF}=0$, can be written as
\begin{equation}\label{MaxWei}
\begin{aligned}
0=\llb W \rrb - \bn^{\top}\llb \bF\rrb^\top\langle \bm{\sigma} (\bF)\rangle\bn&=\llb W \rrb - \bn^{\top}\bn\bm{a}^{\top}\langle \bm{\sigma} (\bF)\rangle\bn\\
&=\llb W \rrb - |\bn|^2\Tr\big(\bm{a}\otimes\langle \bm{\sigma} (\bF)\rangle\bn\big),\\
&=\llb W \rrb - \Tr\big((\bm{a}\otimes\bn)\langle \bm{\sigma} (\bF)\rangle^{\top}\big),\\
&=\llb W \rrb - \langle \bm{\sigma} (\bF)\rangle:\llb \bF \rrb,
\end{aligned}
\end{equation}
where the symbol ``$:$'' denotes the Frobenius inner product of two matrices, $\bF_1:\bF_2:=\Tr(\bF_1\bF_2^{\top})$. The last equation above is known as the second Weierstrass-Erdmann condition.  An additional jump condition was derived in \cite{GrTr11} by making use of the Weierstrass convexity condition \eqref{Weconv}. It is given by 
\begin{equation}
\label{halfgra2}
\llb \bm{\sigma}(\bF) \rrb^{\top}\bm{a} = 0.
\end{equation}
Even though the above condition is also satisfied for the configuration we will analyze, it plays no role in the dynamical stability analysis. 
\begin{remark}
\label{ReGT2}
The study of necessary conditions for identifying singular minimizers that exhibit surfaces of gradient discontinuity (e.g., \eqref{eq:twin}) is one of the main goals in the framework of the Calculus of Variations.  Equations \eqref{halfgra1}, \eqref{MaxWei} and \eqref{halfgra2} are widely known as necessary conditions on the surface of jump discontinuity of the gradient of any strong local minimizer of the energy defined in \eqref{elasen1} (see Appendix \ref{GrabTro}; cf. \cite{GrTr19}). It is worth mentioning that these equations, together with \eqref{eq:rank-one},  define a codimension one variety in $\R_{+}^{2\times 2}$  called the jump set $\cJ$ (see Eq. \eqref{Jump3}; cf. \cite{GrTr16,GrTr19}). It describes the set of homogeneous deformation gradients $\bF$ that permit energy-neutral  nucleation of layers of a new phase with a compatible deformation gradient $\bF + \bm{a}\otimes\bn$ (see \cite{GrTr19}). In this context, all (static) phase boundaries satisfying the Rankine-Hugoniot conditions \eqref{RHe1}, the Maxwell rule and Eq. \eqref{halfgra2} lie within the jump set. The notion of (variational) stability, characterized  by quasiconvexity \eqref{Bin1}, constitutes an additional necessary condition for strong local minimizers \cite{Tah02}. Roughly speaking, a given deformation is stable if  the energy cannot be lowered by any admissible variation of the deformation. In fact, the main idea behind the method developed in \cite{GrTr16} to derive explicit necessary conditions for (static) phase boundaries \eqref{eq:twin} to be strong local minimizers, is to identify a stable part of the jump set;  namely, gradients $\bF$ within the jump set for which quasiconvexity holds. Since arbitrary small neighborhoods of these gradients  contain points where quasiconvexity fails to hold, it follows that such gradients  lie on the so called elastic \textit{binodal} \cite{GrTr19b}, defined as the boundary of the region of failure of quasiconvexity (see Appendix \ref{Bin1}). The binodal is of fundamental importance in nonlinear elasticity \cite{GrTr13,GrTr19b}; it is, however, difficult to compute explicitly because of the lack of an algebraic characterization for quasiconvexity. The method proposed by Grabovsky and Truskinovsky \cite{GrTr16,GrTr19b,GrTr19}, described above, overcomes this practical limitation by characterizing a portion of the binodal as a stable subset of the jump set.  For completeness, Appendix \ref{GrabTro} contains a summary of the  conditions for phase coexistence associated with the jump set, whereas Appendix \ref{L-Hgt} provides the Legendre-Hadamard conditions for Hadamard materials that characterize the stable part of the jump set as derived in \cite{GrTr16}. It is worth noting that, in a subsequent work \cite{GrTr19}, Grabovsky and Truskinovsky further establish that for Hadamard materials with large enough shear modulus, the jump set delivers the entire binodal. This result implies that all nontrivial configurations minimizing the associated energy functional are simple laminates.
\end{remark}

\subsection{Stability of subsonic elastic phase boundaries}
\label{secstabtheory}

Consider the subsonic phase boundary \eqref{shock1} (weak solution to \eqref{model}), for which jump conditions \eqref{RHe1} together with an already prescribed regular kinetic jump condition of the form \eqref{King0}, hold. In addition to the rank-one convexity condition for $W$ at  $\bF=\bFs$, we also assume that the constant multiplicity assumption of M\'etivier \cite{Me4} is satisfied (see Section \ref{model1}). These hypotheses allow the dynamical stability analysis of Freist\"{u}hler and Plaza \cite{FrP1} to be applied to any subsonic phase boundary of the form $(\bFs,\bm{v}^{\pm},s,\bn)$, interpreted as an undercompressive shock.

The dynamical stability analysis of \eqref{shock1} involves a Fourier-Laplace decomposition of the constant-coefficient linearized problem associated with \eqref{HSCL} at the configuration \eqref{shock1} (see \cite{BS} for details). By considering single normal modes of the form $\tilde{u} \sim e^{\lambda t} e^{\ii \xi \cdot \bm{x}}$ with spatio-temporal frequencies lying in the set
\begin{equation}
\label{genfreq}
\Gamma_{\bn}^+ = \left\{ (\lambda,\bm{\xi}) \in \C \times \R^d \, : \, \Re \lambda > 0, \;\, \bm{\xi} \cdot \bn = 0, \,\; |\lambda|^2 + |\bm{\xi}|^2 = 1 \right\},
\end{equation}
as solutions to the linearized problem around the phase boundary $(\bFs,\bm{v}^{\pm},s,\bn)$ (see \cite{FrP1}), one arrives at the \emph{Lopatinski\u{\i} determinant} or stability function
\[
\Delta:\Gamma_{\bn}^+\to\C,
\]
\begin{equation}
\label{gLopdet}
\Delta(\lambda,\bm{\xi}) = \det \begin{pmatrix}
\widetilde{\mathcal{R}}^s(\bF^{-}) & \widehat{\mathcal{Q}} & \widetilde{\mathcal{R}}^{u}(\bF^+) \\
\widetilde{p}^- & \widehat{q} & \widetilde{p}^+
\end{pmatrix},
\end{equation}
where
\[
\widehat{\mathcal{Q}}(\lambda,\bm{\xi}) = 
\begin{pmatrix}
\llb \bF\rrb \bn\vspace{0.2cm}\\
-\lambda s\llb \bF\rrb \bn-\ii\llb\bm{\sigma}(\bF)\rrb\xi
\end{pmatrix}
\in \mathbb{C}^{2d \times 1},
\]
\[
\widehat{q}(\lambda,\bm{\xi}) = -\lambda (D_s g) + \ii( \bm{\xi} \cdot D_{\bn} g) \in \C^{1 \times 1},
\]
\[
\widetilde{p}^-(\lambda,\bm{\xi}) = -(D_{(\bF^-, \bm{v}^-)} g) \mathcal{K}_{s,\bn} (\bF^-) \widetilde{\mathcal{R}}^{s}(\bF^-) \in \C^{1 \times d},
\]
\[
\widetilde{p}^+(\lambda,\bm{\xi}) = (D_{(\bF^+, \bm{v}^+)} g) \mathcal{K}_{s,\bn}(\bF^+)  \widetilde{\mathcal{R}}^{u}(\bF^+) \in \C^{1 \times d}.
\]
The function $\Delta$ is jointly analytic in $(\lambda,\bm{\xi}) \in \Gamma_{\bn}^+$ and homogeneous of degree one.  Also, by continuity of the eigenprojections, the Lopatinski\u{\i} determinant can be defined for all frequencies within the set 
\[
\Gamma_{\bn} := \left\{ (\lambda,\bm{\xi}) \in \C \times \R^d \, : \,\; \Re \lambda \geq 0, \,\; \bm{\xi} \cdot \bn = 0, \,\; |\lambda|^2 + |\xi|^2 = 1 \right\},
\]
(see \cite{Kre70,M2,M1,Me2} for further information). Here $\widetilde{\mathcal{R}}^{s,u}(\bF) \in \C^{1 \times d}$ represent the right stable and unstable invariant spaces of a matrix field $\M_{s,\bn}(\bF):\Gamma_{\bn}\to\C^{2d\times 2d}$  for each $\bF$ near $\bFs$, and $\mathcal{K}_{s,\bn}(\bF):\Gamma_{\bn}\to\C^{d^2+d\times 2d}$ denote continuous mappings.  $\M_{s,\bn}$ and $\mathcal{K}_{s,\bn}$ are given by explicit formulae in terms of the first and second derivatives of the energy function $W$ (for their precise form in the case of a static interface, see Section \ref{RedLop1}; for the general case of a dynamic phase boundary, the reader is referred to \cite{FrP1}). In Eq. \eqref{gLopdet}, $\widehat{\mathcal{Q}}$ is the ``jump vector" associated to Rankine-Hugoniot conditions \cite{M1}, and $\widehat{q}$ denotes its kinetic counterpart \cite{Fr3}. All elements in Eq. \eqref{gLopdet} are evaluated at the end states $\bF=\bFs$ together with $\bm{v}^{\pm}$, $s$, $\bn$ and depend continuously on its parameters for all subsonic $s$ including $s=0$. Furthermore, as a function of $(\lambda,\bm{\xi})$, $\Delta$ is analytic in $\Gamma^+_{\bn}$  and continuous in the whole set $\Gamma_{\bn}$. 

The stability function or \emph{Lopatinski\u{\i} determinant} $\Delta$ determines the solvability of the linearized problem around $(\bFs,\bm{v}^{\pm},s,\bn)$ by wave solutions that violate an $L^2$ well-posedness estimate. A zero of $\Delta$ implies the existence of spatially decaying solutions with temporal growth rate $\exp(t\:\Re\lambda)$. Thus, a necessary condition for well-posedness of the linearized problem is that $\Delta$ does not vanish in the open set $\Gamma_{\bn}^+$. In this case, we say that $\Delta$ satisfy the \emph{weak Lopatinski\u{\i} condition} and  $(\bFs,\bm{v}^{\pm},s,\bn)$ is termed \textit{weakly stable}. A stronger condition, known as the \emph{uniform Lopatinski\u{\i} condition}, requires $\Delta$ not to vanish in the whole frequency set $\Gamma_{\bn}$ (allowing time frequencies with $\Re \lambda = 0$) and it is a sufficient condition for the well-posedness of the nonlinear system, as the analyses of Majda \cite{M2,M1} and M\'etivier \cite{Me1,Me2} show. In this case, the phase boundary $(\bFs,\bm{v}^{\pm},s,\bn)$ is called \emph{uniformly stable}. Finally, if $\Delta$ has a zero $(\lambda,\bm{\xi})$ in $\Gamma_{\bn}^+$ (with $\Re \lambda > 0$) the phase boundary is referred to as \emph{strongly unstable}. The theory is thus applicable to static phase boundaries of the form $(\bFs,\bm{0},0,\bn)$ under multidimensional perturbations that depend on time.

\section{Static phase boundaries for compressible Hadamard materials}
\label{sechada}

This section describes a family of phase boundaries \eqref{shock} for compressible materials of Hadamard type (cf. \cite{Hay68,Joh66}), characterized by stored-energy functions of the form
\begin{equation}
\label{Hadamardmat0}
W(\bF) = \frac{\mu}{2} \Tr (\bF^\top \bF) + h(\det \bF).
\end{equation}
As stated above, $\Tr(\cdot)$ denotes the trace function, $\mu > 0$ is a positive constant known as the shear modulus and $h = h(J)$, for $J = \det \bF \in (0,\infty)$ is a real valued function of class $C^2$. $h$ is known as the volumetric density function and accounts for energy changes due to changes in volume. From its definition, one may verify that function \eqref{Hadamardmat0} satisfies the principles of frame indifference, material symmetry and objectivity. For a discussion on the physical model and its main properties, see \cite{PlVa22}. Following Grabovsky and Truskinovsky \cite{GrTr16}, we concretely take
 \begin{equation}
 \label{hdef1}
h(J)=\frac{1}{4J_0}\Big((J_0-J)^2-1\Big)^2,
\end{equation} 
where $J_0>1$. The failure of quasiconvexity of $W$ follows from the non-convex behavior of  $h$ \cite{GrTr19}, which attains its minima at the points $J=J_0\pm1$. Concretely, $h$ exhibits the double-well shape property \cite{GrTr16} whenever $J_0>1$\footnote{Formula \eqref{hdef1} is only required to hold within an arbitrary neighborhood of $[J_0-1,J_0+1]$. $h(J)$ can be modified outside this neighborhood (as long as it agrees with its convex hull $h^{**}$ there \cite{GrTr19}) in order to meet the requirement that infinite compression costs infinite energy; that is, $h(J)\to\infty$, as $J\to0^+$. Thus, $J_0>1$ is not a restriction.} . This is a fundamental assumption for the study of variational stability of phase boundaries in Hadamard materials \cite{GrTr16,GrTr19}. The configurations we will analyze twins two pure stretch deformation gradients given by
\begin{equation}
\label{twins0}
\Fm:=\begin{pmatrix}
\big(J_0-\frac{1}{\sqrt{3}}\big)\frac{1}{\sqrt{\theta_0}}& 0\vspace{.3cm}\\
0 & \sqrt{\theta_0}
\end{pmatrix},\quad\quad  
\Fp:=\begin{pmatrix}
\big(J_0+\frac{1}{\sqrt{3}}\big)\frac{1}{\sqrt{\theta_0}} & 0\vspace{.3cm}\\
0 & \sqrt{\theta_0}
\end{pmatrix},
\end{equation}
where
\begin{equation}\label{teta0} 
\theta_0=\frac{3\mu J_0}{2}.
\end{equation} 
$\bFu$ are clearly rank-one connected (see \eqref{eq:rank-one}) with $\bm{a}=\tfrac{2}{\sqrt{3\theta_0}}\eu$ and $\bn=\eu$, since
\begin{equation}\label{HadJu1}
\Fp-\Fm=\frac{2}{\sqrt{3\theta_0}}\eu\otimes\eu.
\end{equation}
Moreover, $J_0>1$ implies that
\begin{equation}
\label{detpos}
J^{\pm}=\det \bFs=\big(J_0\pm \frac{1}{\sqrt{3}}\big)>0.
\end{equation} 
The rotationally invariant energy deformations (states) are defined by 
\[
\mathscr{U}^+=\text{SO(2)}\Fp, \qquad  \mathscr{U}^-=\text{SO(2)}\Fm,
\]
where
\[
\text{SO(2)}=\{\bm{Q}\in\R^{2\times 2}: \bm{Q}^{\top}\bm{Q}=\Ido,\; \det \bm{Q}=1\},
\]
is the group of proper rotations in the plane. It should be emphasized that the states $\bFu$ are not local minima of $W$, given that $\bm{\sigma}(\bFu)\neq0$ (see Eqs. \eqref{Sigtwin} below). Nevertheless, this assumption is not needed for the dynamical stability analysis of Freist\"{u}hler and Plaza (see \cite{FrP1}, Theorem 1). Instead, the relevant property for the present work is that the phase boundary \eqref{shock} defined by  $\bF^{\pm}=\bFu$ and $\bn=\eu$ (namely, $(\bFu,\bm{0},0,\eu)$) belongs to the jump set  (see Remark \ref{Eder1} below) and is stable in the variational sense whenever 
 \[
 \mu>\frac{2}{3}\Big(1+\frac{1}{\sqrt{3}J_0}\Big).
 \]
 This is established in Appendix \ref{L-Hgt}, Section \ref{Exa1}, through the verification of  the generalized Legendre-Hadamard conditions \cite{GrTr16}, which characterize a stable portion of the jump set. In other words, the considered family of phase boundaries  $(\bFu,\bm{0},0,\eu)$, parametrized by $J_0 > 1$ and $\mu > 0$, lies within a stable subset of the jump set for the given range of material parameters.

For hyperelastic materials of Hadamard type characterized by \eqref{Hadamardmat0}-\eqref{hdef1}, we analytically evaluate the Lopatinski\u{\i} function   $\Delta$ associated to the (static) phase boundaries $(\bFu,\bm{0},0,\eu)$. We characterize the dynamical stability for all values of the material parameters $J_0>1$ and $\mu>0$, under both the Maxwell rule and a specific Abeyaratne-Knowles rule, which can be regarded as a perturbation of the former.  

\subsection{Derivatives of the stress}

The first Piola-Kirchhoff stress tensor $\bm{\sigma} =\partial W/\partial \bF $ is given by (see \cite{PlVa22})
\begin{equation}
\label{HadPiolKir}
\bm{\sigma}(\bF) = \mu \bF + h'(J) \, \Cof \bF , \qquad \bF \in \R_+^{d\times d}.
\end{equation}

On the other hand, in dimension $d =2$ the matrices  $\bB_i^j$ containing the  second derivatives of the stored-energy function \eqref{Hadamardmat0} are given by (see Corollary 2.8 in \cite{PlVa22})
\begin{equation}
\label{exprBijs}
\begin{aligned}
\bB_{i}^{i} (\bF)&=\mu \mathbb{I}_2+h''(J)\Big((\Cof \bF)_{i}\otimes(\Cof \bF)_{i}\Big),\qquad i=1,2\\
\bB_{1}^{2}(\bF)&=h''(J)\Big((\Cof \bF)_{2}\otimes(\Cof \bF)_{1}\Big)+ h'(J)(\ed\otimes \eu - \eu\otimes \ed),\\
\bB_{2}^{1}(\bF)&= \bB_{1}^{2}(\bF) ^\top,
\end{aligned}
\end{equation}
for each $\bF \in \R_+^{2\times 2}$. Using the first and second derivative of the volumetric density function \eqref{hdef1},
$$h'(J):=\frac{1}{J_0}\big((J-J_0)^3-(J-J_0)\big),\qquad h''(J):=\frac{1}{J_0}\big(3(J-J_0)^2-1\big),$$
we can evaluate $\bB_{i}^{j}$ at the deformation states $\bFu$ defined in \eqref{twins0}. Using \eqref{detpos}, we readily obtain $h''(J^{\pm})=0$, which implies 
\begin{equation}
\label{BEpm}
{\bB^i_i}^{\pm}=\mu\Ido,\quad (i=1,2), \qquad {\bB_1^2}^{\pm}=({\bB_2^1}^{\pm})^{\top}=\pm h_0\begin{pmatrix} 0 & -1\\\ 1 & 0 \end{pmatrix},
\end{equation}
where 
\begin{equation}\label{ache0}
h_0:=h'(J^{+})=\frac{-2}{3J_0\sqrt{3}}.
\end{equation} 
From the relations \eqref{BEpm} and \eqref{defacouten}, the acoustic tensor at $\bm{\xi}=\eu$ and $\bF=\bFs$ is given by
\[
\bN(\eu,\bFu)={\bB^1_1}(\bFs)=\mu\Ido.
\]
Therefore $\min\{\text{eigenvalues of}\:\bN(\eu,\bFu)\}=\mu$ and, according to \eqref{subSp1}, our phase boundary is subsonic for any traveling speed $s$ such that $0\leq s^2<\mu$; consequently, the stability analysis of \cite{FrP1} applies.

\subsection{Rank-one convexity and the constant multiplicity assumption}

In this Section, we verify the rank-one convexity property at the end states $\bFu$  (local hyperbolicity) and the constant multiplicity assumption (see Section \ref{model1}). Due to the double-well form of $h$ (see \eqref{hdef1}), $W$ becomes non-convex, giving rise to the \textit{spinodal region}, defined as the set of states $(\bF,\bm{v})\in\cU$ at which  the  Legendre Hadamard condition fails to hold.  

Let us find the region $\cU_H$ of states $(\bF,\bm{v})\in\cU$, where the Legendre-Hadamard condition,
\begin{equation}
\label{LHcond}
\bm{\eta}^{\top} \bN(\bm{\xi},\bF) \bm{\eta} > 0, \quad \text{for all } \; \bm{\xi}, \bm{\eta} \in \R^d \backslash \{0\},
\end{equation}
holds for $W$. $\cU_H$ is also called the region of \textit{hyperbolicity}. Since it is an open set,  we only need to check that $\bFu\in\cU_H$ in order to ensure rank-one convexity at each state. This, in turn, implies that the system \eqref{model} is hyperbolic in open neighborhoods of $\bFu$ (see, e.g., \cite{Da4e}).  For the general class of compressible Hadamard materials characterized by \eqref{Hadamardmat0}, the acoustic tensor and its eigenvalues can be explicitly calculated for all dimension $d\geq2$ (cf. \cite{PlVa22}). It is given by 
\begin{equation}
\label{exprAT}
\bN(\bm{\xi},\bF) = \mu |\bm{\xi}|^2 \Id + h''(J) \Big( \big( (\Cof \bF) \bm{\xi} \big) \otimes \big( (\Cof \bF) \bm{\xi} \big) \Big),
\end{equation}
for $\bm{\xi} \in \R^d$, $\bm{\xi} \neq 0$, $\bF \in \R_+^{d\times d}$,
with associated eigenvalues 
\begin{equation}
\label{defkaps}
\begin{aligned}
\kappa_1(\bm{\xi},\bF)&:= \mu|\bm{\xi}|^2,\\
\kappa_2(\bm{\xi},\bF)& := \mu|\bm{\xi}|^2 + h''(J) \big| (\Cof \bF)\bm{\xi} \big|^2,
\end{aligned}
\end{equation}
and with algebraic multiplicities $m_1 = d-1$ and $m_2 = 1$, respectively. Thus,  $W$ in \eqref{Hadamardmat0} satisfies  the constant multiplicity property at any state $\bF$ and for any function $h$ of that form. In terms of \eqref{defkaps}, the Legendre-Hadamard condition \eqref{LHcond} is tantamount to the positivity of each eigenvalue for all $\bm{\xi}\neq0$. Notice that this condition holds trivially at any state $\bF$ if $h''(J)>0$ in $(0,\infty)$; this is not the case, however, for the function $h$ considered in \eqref{hdef1}. Thanks to the explicit formulae of the eigenvalues of the acoustic tensor,  $\cU_H$  can be computed explicitly for any smooth function $h$. The first eigenvalue in \eqref{defkaps} is always positive; however, the positivity of the second one might fail to hold when $h''(J)<0$. To investigate the sign of $\kappa_2$, note that it can be recast as
\[
\kappa_2(\bm{\xi},\bF) := \bm{\xi}^{\top}\Big(\mu \Id+ h''(J) (\Cof \bF)^{\top}\Cof \bF\Big)\bm{\xi}.
\]
Thus, the sign of $\kappa_2$ as a function of $\bm{\xi}$ is determined by the symmetric matrix $\mu \Id+ h''(J) (\Cof \bF)^{\top}\Cof \bF$. The eigenvalues of this matrix have the form $\mu+h''(J)\theta$, where $\theta$ is any eigenvalue of $(\Cof \bF)^{\top}\Cof \bF$. Thus, if $\theta_{\rm{max}}$ denotes  the greatest eigenvalue of matrix $(\Cof \bF)^{\top}\Cof \bF=\Cof(\bF^{\top}\bF)$, then the positivity of $\kappa_2$, and therefore the Legendre-Hadamard condition, is equivalent to the condition $\mu+h''(J)\theta_{\rm{max}}>0$. Hence, the hyperbolicity region has the general form
\[
\cU_{H} = \Big\{ (\bF,\bm{v}) \in \R^{d \times d} \times \R^d \, : \, J=\det \bF > 0, \,\; \mu+h''(J)\theta_{\rm{max}}>0\Big\},
\]
where $\theta_{\rm{max}}=\theta_{\rm{max}}(\bF)$ is the greatest eigenvalue of  $\Cof (\bF^{\top} \bF)$.  Note that the last definition includes the case $h''(J)>0$, given that  all eigenvalues of $(\Cof \bF)^{\top}\Cof \bF=\Cof(\bF^{\top}\bF)$ are positive ($J>0$ implies non zero eigenvalues). Observe that, for a given function $h$, $\cU_{H}$ is expressed solely in terms of the deformation gradient $\bF$, through $\det \bF=J$ and the largest eigenvalue of $\Cof(\bF^{\top}\bF)$. Thus, with a slight abuse of notation we write $\cU_{H}=\cU_{H}(\bF)$. Moreover, $\cU_{H}$ can be characterized in terms of all the eigenvalues, $\theta_1,\cdots,\theta_d$, of $\Cof(\bF^{\top}\bF)$, in view that
\begin{equation}
\label{cofdet1}
\prod_{k} \theta_k=\det \big(\Cof(\bF^{\top}\bF)\big)= J^{2(d-1)}.
\end{equation}
Indeed, for $d=2$ and $h$ defined by \eqref{hdef1}, assume $\theta$, $\theta'$ are the two (positive) eigenvalues of $\Cof(\bF^{\top}\bF)$. From \eqref{cofdet1} we have  $J=\sqrt{\theta}\sqrt{\theta'}$ and since 
\[
h''(J)=\frac{1}{J_0}\big(3(J-J_0)^2-1\big),
\]
the hyperbolicity region associated to the selected function \eqref{hdef1} can be written in terms of $\theta,\theta'$ as
$$\cU_{H}(\bF)=\cU_{H}(\theta,\theta'):= \Big\{ (\theta,\theta') \in \R^+\times \R^+ \, : \,  \mu+\tfrac{\widetilde{\theta}}{J_0}\big(3(\sqrt{\theta}\sqrt{\theta'}-J_0)-1\big)^2>0\Big\},$$
where $\widetilde{\theta}=\max\{\theta,\theta' \}$. 

Figure \ref{figcontour1} shows  $\cU_{H}$ (shaded region) in the $(\theta, \theta')$-plane for two pairs of material parameters, $\mu=1$, $J_0=1$ (panel (a)), and $\mu=1.2$, $J_0=1.3$ (panel (b)). The dashed lines represent the points where $h''(J)=0$ (or equivalently, where $J=\sqrt{\theta}\sqrt{\theta'}$), so these lines bound the regions where  $h''$ is positive and negative. Observe that there are points $(\theta,\theta')$ in the hyperbolicity region  for which $h''<0$ (interior region bounded by dashed lines). Finally, it is easy to verify that  $h''$ vanishes at  $\bFu$ (that is, $h''(J^{\pm})=0$) by direct substitution of \eqref{detpos}. That is, $\bFu\in\cU_H$ and then the Legendre-Hadamard condition holds. It also implies that the system in \eqref{HSCL} is hyperbolic in open neighborhoods of the two end states $\bFu$.
    
\begin{figure}[t]
\begin{center}
\subfigure[$\mu=1$, $J_0=1$]{\label{hipReg1}\includegraphics[scale=.75, clip=true]{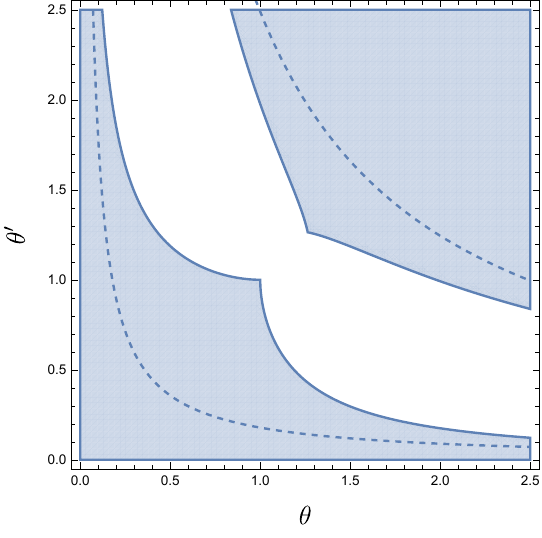}}\quad\qquad
\subfigure[$\mu=1.2$, $J_0=1.3$]{\label{hipReg2}\includegraphics[scale=.75, clip=true]{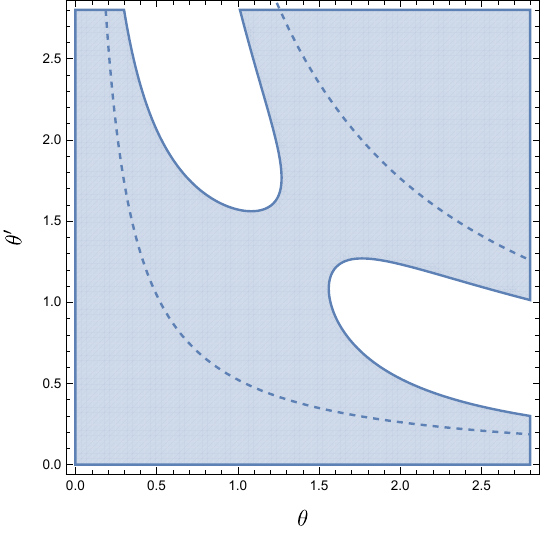}}
\end{center}
\caption{\small{Illustration of the hyperbolicity region $\cU_{H}$ (shaded region) for two pairs of values of the material parameters $J_0,\mu$. The spinodal region is contained in the region where $h''(J)<0$ (connected region bounded by the dashed lines).}}
\label{figcontour1}
\end{figure}

\subsection{Jump conditions across the interface}
\label{JumpC1}

In this Section we verify that the phase boundary $(\bFu,\bm{0},0,\eu)$ satisfies the Rankine-Hugoniot jump conditions \eqref{RHe1} and the jump conditions associated to both kinetic rules, \eqref{King0} and \eqref{MaxRu}. By substituting $\bFu$, $\bm{v}^{+}=\bm{v}^{-}=0$, $s=0$ and $\bn=\eu$ into \eqref{RHe1}, one verifies that the first equation  is trivially satisfied, whereas the third one, which in dimension $d=2$ simplifies to (see Remark \ref{crosrel})
\[
-\llb \bF\rrb\eu^{\perp}=\bm{0},
\]
is likewise verified. On the other hand, substituting \eqref{twins0} into \eqref{HadPiolKir} and taking into account that $h'(J^+)=h_0=-h'(J^{-})$, we obtain
\begin{equation}\label{Sigtwin}
\begin{aligned}
&\bm{\sigma}(\Fm):=\frac{1}{\sqrt{\theta_0}}\begin{pmatrix}
\mu J_0-\frac{\mu}{\sqrt{3}}-h_0\theta_0 & 0\\
0 & -h_0 J_0+\frac{h_0}{\sqrt{3}}+\mu\theta_0
\end{pmatrix},\\
&\bm{\sigma}(\Fp):=\frac{1}{\sqrt{\theta_0}}\begin{pmatrix}
\mu J_0+\frac{\mu}{\sqrt{3}}+h_0\theta_0 & 0\\
0 & h_0 J_0+\frac{h_0}{\sqrt{3}}+\mu\theta_0
\end{pmatrix},
\end{aligned}
\end{equation}
and therefore
\begin{equation}
\label{sigjum1}
\begin{aligned}
\llb \bm{\sigma}(\bF) \rrb = \bm{\sigma}(\Fp)-\bm{\sigma}(\Fm) &= \frac{1}{\sqrt{\theta_0}}\begin{pmatrix}
\frac{2\mu}{\sqrt{3}}+2h_0\theta_0 & 0\\
0 & 2J_0h_0
\end{pmatrix}\\
&=\frac{2J_0 h_0}{\sqrt{\theta_0}}\begin{pmatrix} 
0 & 0\\
0 & 1
\end{pmatrix},
\end{aligned}
\end{equation}
in view that $\theta_0=\frac{3}{2}\mu J_0$ and $h_0 = -2/(3J_0\sqrt{3})$. Consequently, one readily finds that $\llb \bm{\sigma}(\bF) \rrb\bn=\llb \bm{\sigma}(\bF) \rrb\eu=0$ and the second equation in \eqref{RHe1} also holds. 

Now, let us verify the kinetic rules; substitute \eqref{detpos} into \eqref{hdef1} to obtain $h(J^+)=h(J^{-})=1/(9J_0)$, which implies
\[
\llb W\rrb = W(\Fp)-W(\Fm) = \frac{2J_0\mu }{\sqrt{3}\theta_0},
\]
and
\[
\llb \bF\rrb^\top\langle \bm{\sigma} (\bF)\rangle = (\Fp-\Fm)^{\top}\frac{1}{2}(\bm{\sigma}(\Fp)+\bm{\sigma}(\Fm)) 
=\frac{2J_0\mu }{\sqrt{3}\theta_0}\begin{pmatrix} 
1 & 0\\
0 & 0
\end{pmatrix}.
\]
By combining the last two expressions, it immediately follows that, for $\bn=\eu$, one has
\[
\llb W \rrb - \bn^{\top}\llb \bF\rrb^\top\langle \bm{\sigma} (\bF)\rangle\bn=0.
\]
That is, the Maxwell rule \eqref{King0} (with $\tilde{g}=0$) is satisfied at $(\bFu,\bm{0},0,\eu)$. Notice that the Abeyaratne-Knowles rule 
\[
\mathscr{F}+\tilde{g}=0,
\]
is also satisfied at the considered phase boundary, given that $s=0$ and that $\tilde{g}$ satisfies \eqref{AbeKR02}.
 
\begin{remark}
\label{Eder1}
We have shown that the phase boundary $(\bFu,\bm{0},0,\eu)$ satisfies the Rankine-Hugoniot jump conditions and the two kinetic rules under consideration. These conditions, together with the constant multiplicity assumption and the rank one convexity, allow the analysis of Freist\"{u}hler and Plaza \cite{FrP1} to be applied to the study of dynamical stability of  $(\bFu,\bm{0},0,\eu)$. Notice that the additional jump condition \eqref{halfgra2} is also satisfied. Indeed, Eq. \eqref{HadJu1} implies that $\boldsymbol{a}=(2/\sqrt{3\theta_0})\eu$. Substituting this relation and \eqref{sigjum1} into \eqref{halfgra2} yields the result. Hence, in view of Remark \ref{ReGT2} and within the variational framework, the phase boundary $(\bFu,\bm{0},0,\eu)$ belongs to the jump set for all $\mu>0$ and $J_0>1$. In Appendix \ref{L-Hgt}, Section \ref{Exa1}, it is shown that, under the specific choice of the material parameters \eqref{cons3}, the phase boundary $(\bFu,\bm{0},0,\eu)$ satisfies the generalized Legendre-Hadamard conditions derived by Grabovsky and Truskinovsky \cite{GrTr16}, which determine a stable portion of the jump set.
\end{remark}

\subsection{Derivatives of the kinetic rule}
In this section we gather the derivatives of $\tilde{g}$ and $\mathscr{F}$ from Eqs. \eqref{King2}-\eqref{AbeKR}, with respect to the parameters $\bF^+,\bm{v}^{+}$, $\bF^-,\bm{v}^{-}$, $\bn$ and $s$. They are needed to compute some matrix blocks of the Lopatinski\u{\i} determinant. In \cite{PL2}, these derivatives are computed for dimension $d=3$; here, we provide the formulae corresponding to the two-dimensional case. Following \cite{PL2}, for each $i,j$, we have
\[
\begin{aligned}
\partial_{F_{ij}^{\pm}}\big(\bn^{\top}\llb \bF\rrb^\top\langle \bm{\sigma} (\bF)\rangle\bn\big)&=\pm n_j\sum_{k} n_k\langle \sigma (\bF)_{ik}\rangle+\frac{1}{2}\bn^{\top}\llb \bF\rrb^\top\sum_{l} n_l \partial_{F_{ij}^{\pm}}\big(\bm{\sigma}(\bFs)_l\big)\\
&=\pm n_j\sum_{k} n_k\langle \sigma (\bF)_{ik}\rangle+\frac{1}{2}\sum_{k,l} n_l(\llb \bF\rrb\bn)_k ({\bB_{j}^{l}}^{\pm})_{ki}.
\end{aligned}
\]
Recall $\bm{\sigma}(\bFs)_l$ denotes the $l$-th column of  $\bm{\sigma}(\bFs)$. Hence, in dimension $d=2$, one gets
\[
\begin{aligned}
D_{({\bF}^{\pm}, \bm{v}^{\pm})}\mathscr{F} =
&\Big(
\pm \bm{\sigma}(\bFs)^\top_1 \mp n_1 \bn^\top \langle \bm{\sigma}(\bF) \rangle^\top - \frac{1}{2} \bn^\top \llb \bF\rrb^\top \sum_{j=1}^2 n_j \bB^{j\pm}_1, \\
&\pm \bm{\sigma}(\bFs)^\top_2 \mp n_2 \bn^\top \langle \bm{\sigma}(\bF) \rangle^\top - \frac{1}{2} \bn^\top \llb \bF\rrb^\top \sum_{j=1}^2 n_j \bB^{j\pm}_2,\bm{0}\Big) \in \R^{1 \times 6},
\end{aligned}
\]
where each component is a $1\times 2$ block. Specializing the last formula to $\bn=\eu$, we arrive at
\begin{equation}\label{derivsF}
\begin{aligned}
&D_{(\bF^{+}, \bm{v}^{+})}\mathscr{F} =\Big(\bm{\sigma}(\bF^+)^\top_1 -  \langle \bm{\sigma}(\bF)_1 \rangle^\top - \frac{1}{2}  \llb \bF_1\rrb^\top  \bB^{1 +}_1, {\bm{\sigma}(\bF^{+})_2^{\top}} - \frac{1}{2}  \llb \bF_1\rrb^\top  \bB^{1 +}_2,\bm{0}\Big) \\
&D_{(\bF^{-}, \bm{v}^{-})}\mathscr{F} =\Big(-\bm{\sigma}(\bF^{-})^\top_1 +  \langle \bm{\sigma}(\bF)_1 \rangle^\top - \frac{1}{2}  \llb \bF_1\rrb^\top  \bB^{1 -}_1, -{\bm{\sigma}(\bF^{-})_2^{\top}} - \frac{1}{2}  \llb \bF_1\rrb^\top  \bB^{1 -}_2,\bm{0}\Big).
\end{aligned}
\end{equation}
On the other hand,
\begin{equation}
\label{DeriF1}
D_{\bn}\mathscr{F}=-\bn^{\top}\Big(\llb \bF\rrb^\top\langle \bm{\sigma} (\bF)\rangle+\langle \bm{\sigma} (\bF)\rangle^{\top}\llb \bF\rrb\Big)\in\R^{1\times 2},
\end{equation}
\begin{equation}
\label{DeriF2}
D_{s}\mathscr{F}=0.
\end{equation}
If the Maxwell rule is assumed, we trivially obtain
\begin{equation}
\label{g1Deriv1}
\begin{aligned}
D_{(\bFs,\bm{v}^{\pm})}\tilde{g}=0,\\
D_s \tilde{g}=0,
\end{aligned}
\qquad \text{at}\;\; \big((\Fp,0),(\Fm,0),0,\eu\big).
\end{equation}

We will also consider the simplest, nontrivial Abeyaratne-Knowles rule given by $\tilde{g}=-\varepsilon s$, for a fixed $\varepsilon>0$. This kinetic rule is the easiest way to model energy dissipation at the interface. It implies that 
\begin{equation}
\label{g1Deriv1AN}
\begin{aligned}
D_{(\bFs,\bm{v}^{\pm})}\tilde{g}=0,\\
D_s \tilde{g}=-\varepsilon,
\end{aligned}
\qquad \text{at}\;\; \big((\Fp,0),(\Fm,0),0,\eu\big).
\end{equation}

\section{The reduced Lopatinski\u{\i} function}
\label{RedLop1}

In this Section we compute all the necessary ingredients to assemble the Lopatinski\u{\i} determinant \eqref{gLopdet} associated to $(\bFu,\bm{0},0,\eu)$. In dimension $d=2$ with $\bn = \eu$, the transversal directions ($\bm{\xi}\cdot\bn=0$) are thus given by 
$\bm{\xi} = (0, \xi_2)^\top$. By a slight abuse of notation, we define then  $\xi_2:= \xi \in \R$ so that the set $\Gamma_{\bn}=\Gamma_{\eu}$ of space-time frequencies can be expressed as
\begin{equation}
\label{freqsp1}
\Gamma:=\Gamma_{\eu} = \{ (\lambda, \xi) \in \C \times \R : |\lambda|^2 + |\xi |^2 = 1,  \Re \, \lambda \geq 0 \},
\end{equation}
with interior
\[
\Gamma^+ = \Gamma\cap \{ \Re \, \lambda > 0 \}.
\]
Consider the two continuous matrix fields $\M_\pm : \Gamma \to \C^{4 \times 4}$ defined by
\begin{equation}
\label{MMatrix1}
\M_\pm(\lambda, \xi) := 
\begin{pmatrix}
\M_{11}^\pm & \M_{12}^\pm \\
\M_{21}^\pm & \M_{22}^\pm
\end{pmatrix}, 
\qquad \M_{ij}^\pm : \Gamma \to \mathbb{C}^{2 \times 2},
\end{equation}
with
\[
\begin{aligned}
\M_{11}^\pm &= i \xi (\bB_1^{1\pm})^{-1}  \bB_2^{1\pm} , \\
\M_{12}^\pm &= -(\bB_1^{1\pm})^{-1}, \\
\M_{21}^\pm &= \xi^2 \bB_1^{2\pm}  (\bB_1^{1\pm})^{-1}  \bB_2^{1\pm} - \lambda^2 \Ido -  \xi^2 \bB_2^{2\pm},  \\
\M_{22}^\pm &= i\xi \bB_1^{2\pm} (\bB_1^{1\pm})^{-1}.
\end{aligned}
\]
For a detailed derivation of the elements of $\M$, see Freist\"{u}hler and Plaza \cite{FrP1}. According to \cite{FrP1},  the curl-free constraint \eqref{curlfree} reduces the normal modes analysis for the stability of the considered elastic phase boundaries to a subspace of amplitudes, whose dynamics in the frequency (or Fourier-Laplace) space is described by the matrix fields $\M_\pm$ (that is why it is called the \emph{reduced} Lopatinski\u{\i} determinant). In view of the analysis in \cite{FrP1}, the stability function \eqref{gLopdet} for such a static planar interface takes the form
\begin{equation}
\label{Lopat1}
\Delta(\lambda, \xi) = \det \begin{pmatrix}
\widetilde{\mathcal{R}}^s_{-} & \widehat{\mathcal{Q}} & \widetilde{\mathcal{R}}^u_+ \\
\widetilde{p}^- & \widehat{q} & \widetilde{p}^+
\end{pmatrix},
\end{equation}
where
\begin{equation}\label{pQgorro}
\widehat{\mathcal{Q}} = 
\begin{pmatrix}
\llb \bF_1\rrb\\
- \ii \xi\llb\bm{\sigma}(\bF)_2\rrb
\end{pmatrix}
\in \mathbb{C}^{4 \times 1},
\end{equation}
\begin{equation}\label{pqgorro}
\widehat{q} = -\lambda (D_s g) + \ii \xi\: \ed\cdot D_{\bn} g \in \C^{1 \times 1},
\end{equation}
\begin{equation}\label{ppgorrma}
\widetilde{p}^+ = (D_{(\bF^+, \bm{v}^+)} g) \mathcal{K}_+ (\lambda, \xi) \widetilde{\mathcal{R}}_+^{u} \in \C^{1 \times 2},
\end{equation}
\begin{equation}\label{pqgorrme}
\widetilde{p}^- = -(D_{(\bF^-, \bm{v}^-)} g) \mathcal{K}_{-} (\lambda, \xi) \widetilde{\mathcal{R}}_{-}^{s} \in \C^{1 \times 2},
\end{equation}
\begin{equation}
\mathcal{K}_\pm(\lambda, \xi) =
\begin{pmatrix}
\ii \xi(\bB_1^{1\pm})^{-1} \bB_2^{1\pm}  & -(\bB_1^{1\pm})^{-1} \\
-\ii \xi \Ido & 0 \\
-\lambda \Ido & 0
\end{pmatrix}
\in \C^{6\times 4},
\end{equation}
and  $\widetilde{\mathcal{R}}_+^{u}:=\widetilde{\mathcal{R}}^{u}(\bF^+)(\lambda, \xi)\in \C^{4 \times 4}$ (resp. $\widetilde{\mathcal{R}}_-^{s}$)  denotes the unstable (resp. stable) space  of the matrix field $\M_{+}$ (resp. $\M_{-}$). Recall that $\Delta$ is analytic in $(\lambda, \xi) \in \Gamma^+$ and continuous in $(\lambda, \xi) \in \Gamma$.

\subsection{Stable and unstable modes}

It this Section, explicit formulae for the invariant spaces $\widetilde{\mathcal{R}}_+^{u}$ and $\widetilde{\mathcal{R}}_-^{s}$ are derived for Hadamard materials \eqref{Hadamardmat0}-\eqref{hdef1} and the phase boundary $(\underline{F}^{\pm},\bm{0},0,\eu)$. They are crucial for the derivation of a closed-form  expression for the associated stability function or Lopatinski\u{\i} determinant.  For general hyperelastic materials, explicit expressions for these spaces are generally unavailable (see, e.g., \cite{PL2}). For the sake of simplicity, in the present two-dimensional case let us write
\[
\bN_{\pm}(\eta,\xi):=\bN\big((\eta,\xi)^{\top},\bFs),
\] 
for $\eta\in\C$, to denote the acoustic tensor defined in \eqref{exprAT}. Freist\"uhler and Plaza \cite{FrP1} showed that for $s=0$, the eigenvalues $\beta=-\ii\eta\in\C$ of $\M_{\pm}$ satisfy the following characteristic equation 
\begin{equation}
\label{charEgn0}
\det\big(\bN_{\pm}(\eta,\xi)+\lambda^2\Ido\big)=0,\quad\text{for all}\quad (\lambda,\xi)\in\Gamma.
\end{equation}

Moreover, if we restrict our analysis to $\Gamma^+$, then the eigenvalues are split into stable and unstable modes depending on whether  $\Im\eta<0$ or $\Im\eta>0$, due to hyperbolicity at the end states. This hyperbolic dichotomy in the dynamical systems sense was originally pointed out by Hersh \cite{Her63}. Under the current assumption, there are exactly two unstable and two stable eigenvalues (counting multiplicities) for each $(\lambda,\xi)\in\Gamma^+$, implying that the stable and unstable spaces of both $\M_{\pm}$ have constant dimension $d=2$ for all $\Gamma^+$. For details see \cite{FrP1} and the references therein. From Lemma 6 in the same reference,  each eigenvector of $\M_{\pm}$ associated to an eigenvalue $\beta=-\ii\eta$ has the form
\begin{equation}
\label{eigvec0}
\bm{r} = \begin{pmatrix}
\bm{w} \\ \ii\big(\eta \bB_{1}^{1 \pm}+\xi \bB_{2}^{1 \pm}\big)\bm{w}
\end{pmatrix},
\end{equation}
where 
\begin{equation}\label{kerpm}
\bm{w}\in \ker\big(\bN_{\pm}(\eta,\xi)+\lambda^2\Ido\big).
\end{equation}
Thanks to the explicit formula of the acoustic tensor \eqref{exprAT}, we can compute the $\eta$-modes from Eq. \eqref{charEgn0}. Using \eqref{BEpm}, we can evaluate \eqref{charEgn0} at the end states \eqref{twins0}, obtaining $$(\mu(\eta^2+\xi^2)+\lambda^2)^2=0.$$
Since the equation does not involve $\M_+$ nor $\M_{-}$, it yields the same eigenvalues $\beta=\ii\eta$ for both matrices $\M_{\pm}$. Solving for $\eta$ we obtain
\begin{equation}
\label{eigvec1}
\eta=\pm\ii\tilde{\eta},\qquad \text{where}\;\; \tilde{\eta}(\lambda,\xi)=\sqrt{\frac{\lambda^2}{\mu}+\xi^2},
\end{equation}
each of them with multiplicity equal to two for all $(\lambda,\xi)\in\Gamma^+$. The square root denotes the principal  branch, so the stable ($\Im \eta<0$) and unstable ($\Im \eta>0$) modes are $\eta=-\ii\tilde{\eta}$ and
 $\eta=\ii\tilde{\eta}$, respectively. To compute the associated eigenvectors, we evaluate \eqref{eigvec0} and \eqref{kerpm} at the end states $\bFu$ to find that $\bN_{\pm}(\eta,\xi)+\lambda^2\Ido = \bm{0}$ and, henceforth, that the right eigenvectors associated to $\beta=-\ii\eta$ have the form
 \begin{equation}
 \label{eigvecp0}
\bm{r}^{\pm}(\eta,\bm{w})= \begin{pmatrix}
\bm{w} \\  \begin{pmatrix}
 \mu\eta\ii & \pm \ii\xi h_0\\
 \mp \ii\xi h_0 & \mu\eta \ii
 \end{pmatrix} \bm{w}
\end{pmatrix},\quad\text{where}\;\; \bm{w}\in \ker \bm{0} = \C^2 = \mathrm{span}\{\eu,\ed\}.
\end{equation}
From Eq. \eqref{eigvecp0}, it is easy to verify that the vectors $\bm{r}_1=\bm{r}^{\pm}(-\ii\tilde{\eta},\eu)$ and  $\bm{r}_2=\bm{r}^{\pm}(-\ii\tilde{\eta},\ed)$ are linearly independent for all $(\lambda,\xi)\in\Gamma^+$ and, therefore, that they form a  basis of the stable invariant  space of $\M_{\pm}$ associated to $\eta=-\ii\tilde{\eta}$. The stable bundle  (matrix whose columns are $\bm{r}_1$ and $\bm{r}_2$) is given by
\begin{equation}
\label{stable}
\widetilde{\mathcal{R}}_{\pm}^{s}(\lambda,\xi)= \begin{pmatrix}
1 &  0 \\ 
0 & 1 \\
\mu \tilde{\eta} & \pm \ii \xi h_0\\
\mp \ii \xi h_0 & \mu \tilde{\eta}
\end{pmatrix}.
\end{equation}

Following the same procedure, vectors $\bm{r}_3=\bm{r}^{\pm}(\ii\tilde{\eta},\eu)$ and  $\bm{r}_4=\bm{r}^{\pm}(\ii\tilde{\eta},\ed)$ form a  basis of the unstable invariant  space of $\M_{\pm}$ associated to $\eta=\ii\tilde{\eta}$. Consequently, the unstable bundle is
\begin{equation}
\label{unstable}
\widetilde{\mathcal{R}}_{\pm}^{u}(\lambda,\xi)= \begin{pmatrix}
1 &  0 \\ 
0 & 1 \\
-\mu \tilde{\eta} & \pm \ii \xi h_0\\
\mp \ii \xi h_0 & -\mu \tilde{\eta}
\end{pmatrix}.
\end{equation}

\subsection{Kinetic rule and jump condition blocks}

In this Section, we evaluate all the block matrices in $\Delta$ (see Eq. \eqref{Lopat1}) at the phase boundary $(\bFu,\bm{0},0,\eu)$ for Hadamard materials \eqref{Hadamardmat0}-\eqref{hdef1}, which satisfy Rankine-Hugoniot jump conditions and the kinetic rules \eqref{MaxRu} and \eqref{AbeKR}. First, from \eqref{twins0} and \eqref{sigjum1}, we have
\[
\llb \bF\rrb=\frac{2}{\sqrt{3\theta_0}}\begin{pmatrix} 
1 & 0\\
0 & 0
\end{pmatrix},\qquad 
\llb \bm{\sigma}(\bF)\rrb=\frac{2J_0 h_0}{\sqrt{\theta_0}}\begin{pmatrix} 
0 & 0\\
0 & 1
\end{pmatrix}.
\]
Thus, from \eqref{pQgorro} we get
\begin{equation}\label{Qrem}
\widehat{\mathcal{Q}} = 
\begin{pmatrix}
\llb \bF_1\rrb \vspace{0.3cm}\\
-i\xi\llb\bm{\sigma}(\bF)_2\rrb
\end{pmatrix}=\begin{pmatrix}
\frac{2}{\sqrt{3\theta_0}}\eu
\vspace{0.3cm}\\
\frac{-2\ii J_0 h_0\xi}{\sqrt{\theta_0}}\ed
\end{pmatrix}
\in \mathbb{C}^{4 \times 1}.
\end{equation}
From Eqs. \eqref{derivsF}, \eqref{DeriF1}, \eqref{DeriF2} and \eqref{exprBijs}, we obtain 
\[
\begin{aligned}
(D_s g)|_{s=0} &= D_s(\mathscr{F} + \tilde{g})|_{s=0} = (D_s \tilde{g})|_{s=0}, \\
(D_{\bn} g)|_{s=0} &= D_{\bn}(\mathscr{F} + \tilde{g})|_{s=0} = 0, \\
(D_{(\bFs, \bm{v}^{\pm})} g)|_{s=0} &= D_{(\bFs, \bm{v}^{\pm})} (\mathscr{F}+\tilde{g})|_{s=0} = \left(h_0\sqrt{\theta_0}, 0,0, \tfrac{J_0 h_0\pm\mu\theta_0}{\sqrt{\theta_0}},0,0\right) \in \mathbb{R}^{1 \times 6}.
\end{aligned}
\]
Using Eqs. \eqref{exprBijs} and \eqref{BEpm}, we also compute
\[
\mathcal{K}_{\pm}(\lambda, \xi) = 
\left(
\begin{array}{cc}
\pm\frac{\ii h_0 \xi}{\mu}
\left(
\begin{array}{cc}
0 & 1 \\
-1 & 0  
\end{array}
\right) & -\frac{1}{\mu}\Ido\\
-\ii \xi \Ido & 0 \\
-\lambda \Ido & 0
\end{array}
\right).
\]
Hence, Eqs. \eqref{pqgorro}-\eqref{pqgorrme} evaluated at $s=0$ finally become 
\begin{align}
&\widehat{q} = -\lambda (D_s g) + \ii (0, \xi)^\top (D_{\bn} g)= -\lambda \big(D_s \tilde{g}|_{s=0}\big)\in \C^{1 \times 1},\label{qrem}\\
&\widetilde{p}^+ = (D_{(\bF^+, \bm{v}^+)} g) \mathcal{K}_+ (\lambda, \xi) \widetilde{\mathcal{R}}_+^{u}(\lambda, \xi),\nonumber\\
&\;\quad =\Big(h_0\sqrt{\theta_0}\tilde{\eta},-\dfrac{\ii(J_0 h_0+\mu\theta_0)\xi}{\sqrt{\theta_0}}\Big)\in \C^{1 \times 2},\label{pmarem}\\
&\widetilde{p}^- = -(D_{(\bF^-, \bm{v}^-)} g) \mathcal{K}_{-} (\lambda, \xi) \widetilde{\mathcal{R}}_{-}^{s}(\lambda, \xi),\nonumber\\
&\;\quad=\Big(h_0\sqrt{\theta_0}\tilde{\eta},\dfrac{\ii(J_0 h_0-\mu\theta_0)\xi}{\sqrt{\theta_0}}\Big)\in \C^{1 \times 2}\label{pmerem}.
\end{align}

\section{Evaluating the Lopatinski\u{\i} determinant}
\label{seclopa}

We now have all ingredients to assemble the Lopatinski\u{\i} determinant and to obtain some stability results for the selected phase boundary from it.

\subsection{Lopatinski\u{\i} function and one dimensional stability} 

Substitute \eqref{stable}, \eqref{Qrem}, \eqref{unstable}, \eqref{pmerem}, \eqref{qrem} and \eqref{pmarem} into \eqref{Lopat1} to obtain
\begin{equation}\label{Lopaq1}
\begin{aligned}
\Delta(\lambda,\xi)&=\det \left(
\begin{array}{ccccc}
 1 & 0 & \frac{2}{\sqrt{3} \sqrt{\theta_0}} & 1 & 0 \\
 0 & 1 & 0 & 0 & 1 \\
 \tilde{\eta} \mu  & -i h_0 \xi & 0 & -\tilde{\eta} \mu  & i h_0
   \xi \\
 i h_0 \xi & \tilde{\eta} \mu  & -\frac{2 i J_0 h_0
   \xi}{\sqrt{\theta_0}} & -i h_0 \xi & -\tilde{\eta} \mu  \\
 \tilde{\eta} h_0 \sqrt{\theta_0} & \frac{i \xi (J_0
   h_0-\mu  \theta_0)}{\sqrt{\theta_0}} & -\lambda \big(D_s \tilde{g}|_{s=0}\big) & \tilde{\eta}
   h_0 \sqrt{\theta_0} & -\frac{i \xi (J_0 h_0+\mu 
   \theta_0)}{\sqrt{\theta_0}} \\
\end{array}
\right)\\
   &=-\frac{4}{3 \theta_0} \left(2 \tilde{\eta} h_0 \big(\theta_0\sqrt{3}\; \tilde{\eta}^2\mu^2 + (3\mu  h_0J_0^2 -\sqrt{3} h_0^2
   \theta_0) \xi^2\big)+3 \lambda \theta_0 g_0 (\tilde{\eta}^2 \mu
   ^2-h_0^2 \xi^2)\right)
\end{aligned}
\end{equation}
where  $g_0:=D_s \tilde{g}|_{s=0}$. This is the main algebraic expression for the Lopatinski\u{\i} determinant we shall be working with. It encodes all the information regarding the dynamical stability of the phase boundary $(\bFu,\bm{0},0,\eu)$ (of the form \eqref{shock}).

\begin{lemma}
\label{lelopa1}
The Lopatinski\u{\i} determinant \eqref{Lopaq1} can be written in normalized form as
\begin{equation}\label{Lopaq2}
\begin{aligned}
\widehat{\Delta}(\lambda,\xi)&:=-\frac{3\theta_0}{2\mu^3}\Delta(\lambda,\xi)\\
&=-4\Big(\dfrac{\lambda^2}{\mu}+ \delta_1 \xi^2\Big)\sqrt{\dfrac{\lambda^2}{\mu}+ \xi^2}+9J_0g_0 \lambda \Big( \dfrac{\lambda^2}{\mu}+\delta_2\xi^2\Big),
\end{aligned}
\end{equation}
where
\begin{equation}
\label{defdeltas}
\delta_1:=1-\frac{4}{27J_0^2\mu^2}-\frac{4}{9\mu^2}, \quad \delta_2:=1-\frac{4}{27J_0^2\mu^2}.
\end{equation}
\end{lemma}
\begin{proof}
First, we eliminate the quadratic term $\eta^2\mu^2$ by substituting  $\mu^2\tilde{\eta}^2=\mu\lambda^2+\mu^2\xi^2$ (see \eqref{eigvec1}) into the Lopatinski\u{\i} determinant \eqref{Lopaq1}. Simplifying and rearranging terms we arrive at
\[
\begin{aligned}
\Delta(\lambda,\xi)&=-\frac{4}{3 \theta_0} \left(2 \tilde{\eta} h_0 \big(\mu\theta_0\sqrt{3}\lambda^2+ (\theta_0\mu^2\sqrt{3}+ 3\mu h_0J_0^2 - \theta_0 h_0^2\sqrt{3}) \xi^2\big)\right.\\
   &\qquad\qquad\quad+3 \lambda \theta_0 g_0 \big( \mu\lambda^2+(\mu^2-h_0^2) \xi^2\big)\Big)\\
  &= -\frac{4}{3 \theta_0} \left(2  h_0 \mu^2\theta_0\sqrt{3}\Big(\dfrac{\lambda^2}{\mu}+ \delta_1 \xi^2\Big)\tilde{\eta}+3\lambda  \theta_0 g_0\mu^2 \Big( \dfrac{\lambda^2}{\mu}+\delta_2\xi^2\Big)\right),
\end{aligned}
\]
where
$$\delta_1=\frac{\theta_0\mu^2\sqrt{3}+ 3\mu h_0J_0^2 - \theta_0 h_0^2\sqrt{3}}{\theta_0\mu^2\sqrt{3}}, \quad \delta_2=\frac{\mu^2-h_0^2}{\mu^2}.$$
Upon substitution of $\theta_0=\frac{3\mu J_0}{2}$, $h_0=\tfrac{-2}{3J_0\sqrt{3}}$, one obtains the Lopatinski\u{\i} function in terms of the material parameters $\mu, J_0$: 
\[
\begin{aligned}
\Delta(\lambda,\xi)&= -\frac{4}{3 \theta_0} \left(-2\mu^3\Big(\dfrac{\lambda^2}{\mu}+ \delta_1 \xi^2\Big)\tilde{\eta}+9J_0g_0\frac{\mu^3}{2} \lambda \Big( \dfrac{\lambda^2}{\mu}+\delta_2\xi^2\Big)\right),
\end{aligned}
\]
where
$$\delta_1=1-\frac{4}{27J_0^2\mu^2}-\frac{4}{9\mu^2}, \quad \delta_2=1-\frac{4}{27J_0^2\mu^2}.$$
Substituting \eqref{eigvec1} and taking $\mu^3/2$ as common factor we arrive at the desired expression $\Delta=-\frac{2\mu^3}{3\theta_0}\widehat{\Delta}$. The Lemma is proved.
\end{proof}

As a first consequence of the expression for the Lopatinski\u{\i} determinant \eqref{Lopaq2} we have the following Corollary.

\begin{corollary}[one-dimensional stability]
\label{cor1dstab}
The (static) phase boundary $(\bFu,\bm{0},0,\eu)$ is uniformly stable with respect to one-dimensional perturbations in compressible Hadamard materials characterized by energy density functions of the form \eqref{Hadamardmat0} - \eqref{hdef1}. More precisely, the associated Lopatinski\u{\i} determinant \eqref{Lopaq2} behaves for $\xi = 0$ as
\[
 \widehat{\Delta}(\lambda,0) =  \Big(\frac{-4+9J_0g_0\sqrt{\mu}}{\mu\sqrt{\mu}} \Big)\lambda^3 \neq 0,
\]
for any $(\lambda,0)  \in \Gamma^+$.
\end{corollary}
\begin{proof}
Set $\xi=0$ and note that for every $(\lambda,0)\in\Gamma^+$, there holds that $\lambda\neq0$. Hence, a direct substitution into \eqref{Lopaq2} yields
\[
\widehat{\Delta}(\lambda,0) =  \Big(\frac{-4+9J_0g_0\sqrt{\mu}}{\mu\sqrt{\mu}} \Big)\lambda^3,
\]
which does not vanish, because for the Maxwell kinetic rule we have $\tilde{g}\equiv0$ and then $g_0=D_s \tilde{g}|_{s=0}=0$. In the case of the Abeyaratne-Knowles rule we have $g_0<0$ (see \eqref{AbeKR}) and, once again, $\widehat{\Delta}(\lambda,0) \neq 0$. This shows the result.
\end{proof}

\subsection{The Maxwell rule}

In this Subsection, we analyze the Lopatinski\u{\i} determinant \eqref{Lopaq2} under the Maxwell (or Hugoniot) rule, corresponding to $\tilde{g}\equiv0$. Since $g_0=D_s \tilde{g}|_{s=0}=0$, the Lopatinski\u{\i} determinant \eqref{Lopaq2} takes the form
 \begin{equation}\label{Lopg0}
\widehat{\Delta}(\lambda,\xi)=-4\Big(\dfrac{\lambda^2}{\mu}+ \delta_1 \xi^2\Big)\sqrt{\dfrac{\lambda^2}{\mu}+ \xi^2}.
\end{equation}
It is easy to see that if $\delta_1\geq0$ then all roots of $\widehat{\Delta}$ have the form 
\begin{equation}\label{rootD1}
(\ii\rho\xi,\xi),\qquad \xi\in\R,
\end{equation}
where $\rho=\pm\sqrt{\mu}$ or $\rho=\pm\sqrt{\delta_1\mu}$. Therefore, $\widehat{\Delta}$ does not vanish for $\Re\lambda>0$, given that all its roots  are purely imaginary in $\lambda$. In particular, $\widehat{\Delta}$ does not vanish in $\Gamma^+$ but it does vanish in $\Gamma$ (at purely imaginary values in $\lambda$). Indeed,  it easy to verify  that the root $(\ii\rho\xi_0,\xi_0)\in\Gamma\setminus\Gamma^+$ for 
\[
\xi_0:=\pm\frac{1}{\sqrt{1+\rho^2}}.
\]

This implies that, under the Maxwell rule with $\delta_1\geq0$, $\widehat{\Delta}$ satisfies the Lopatinski\u{\i} condition (but not uniformly), and therefore the phase boundary $(\bFu,\bm{0},0,\eu)$ is weakly stable as long as $\delta_1\geq0$. Conversely,  when $\delta_1<0$, $\widehat{\Delta}$ admits not only  purely imaginary roots in $\lambda$,  but also roots with $\lambda$ real and positive of the form $$\big(\sqrt{-\delta_1\mu}\:\xi,\xi\big),\qquad \xi>0,$$ 
which arise from the second-order polynomial multiplying the square root in \eqref{Lopg0}. As in the last case,  one may find a value of $\xi$ such that the associated root lies in $\Gamma^+$; that is, $\widehat{\Delta}$ vanishes for at least one point in $\Gamma^+$, implying that the phase boundary is strongly unstable. Thus, we have proved the following Lemma.

\begin{lemma}
\label{PreCrit1}
Consider the (static) phase boundary $(\bFu,\bm{0},0,\eu)$ (of the form \eqref{shock}) in a compressible Hadamard material characterized by \eqref{Hadamardmat0} and \eqref{hdef1}, with $J_0>1$, $\mu>0$. Let $\delta_1 = \delta_1(\mu,J_0)$ be the parameter defined in \eqref{defdeltas}. Then the associated Lopatinski\u{\i} determinant under the Maxwell rule \eqref{Lopg0} behaves as follows:
\begin{itemize}
\item [\rm{(i)}] If $\delta_1 \geq0$ then $\widehat{\Delta}$ does not vanish in $\Gamma^+$ but does vanish in $\Gamma$ (the phase boundary is weakly stable).
\item [\rm{(ii)}] If $\delta_1<0$ then $\widehat{\Delta}$ vanishes in at least one point in $\Gamma^+$ (the phase boundary is strongly unstable)
\end{itemize}
\end{lemma}

For the sake of clarity, we express the last Lemma in terms of the original material parameters $\mu>0$ and $J_0>1$; we just have to determine the sign of $\delta_1=\delta_1(\mu,J_0)$, defined in \eqref{defdeltas}, as a function of $\mu$, $J_0$. To that end, let us consider the function  $\mu \mapsto \delta_1(\mu)$ for $J_0\in[1,\infty)$ fixed. It is a strictly increasing real function on $(0,\infty)$, given that its derivative 
 \[
 \delta_1'=\frac{d}{d\mu}\delta_1=\frac{8}{9\mu^3}+\frac{8}{27J_0^2\mu^3},
 \]
 is positive for all $\mu>0$. Furthermore, one finds that $\delta_1$ admits a unique positive zero
 \begin{equation}
 \label{mu:barr}
 \mu = \mu_*:=\frac{2}{3}\sqrt{1+\frac{1}{3J_0^2}}.
 \end{equation}
Therefore, $\delta_1<0$ for all $0<\mu< \mu_*$ by monotonicity; hence, Lemma \ref{PreCrit1} (ii) ensures that $\widehat{\Delta}$ vanishes at some point in $\Gamma^+$, implying that the phase boundary is strongly unstable. On the other hand,  if $\mu\geq \mu_*$ then $\delta_1\geq0$, so $\widehat{\Delta}$ vanishes on $\Gamma$ only at pure imaginary roots in $\lambda$ and then the phase boundary is weakly stable.  As a consequence,  $\mu_*$ turns out to be a transition point from strong instability to weak stability. We summarize our findings into the following main result.

\begin{theo}[stability criteria under the Maxwell rule]\label{MaxSta}
Consider the (static) phase boundary $(\bFu,\bm{0},0,\eu)$ in a compressible Hadamard material defined by \eqref{Hadamardmat0}, \eqref{hdef1}, with $J_0>1$, $\mu>0$.
\begin{itemize}
\item[\rm{(i)}] If $J_0>1$ and $0<\mu<\mu_*$ then $\delta_1<0$ and the phase boundary is  \emph{strongly unstable}.
\vspace{.2cm}
\item[\rm{(ii)}] If $J_0>1$ and $\mu\geq\mu_*$ then $\delta_1\geq0$ and the phase boundary  is \emph{weakly stable}.
\end{itemize}
\end{theo}

\subsection{Abeyaratne-Knowles kinetic rule}

In what follows, we adopt a kinetic relation of linear type, as first proposed by Abeyaratne and Knowles to describe irreversible processes close to thermodynamic equilibrium. It has the form
\begin{equation}\label{AbeKR0}
\tilde{g}=-\varepsilon s,\quad \varepsilon>0.
\end{equation}
Since $D_s\tilde{g}=-\varepsilon<0$, then \eqref{AbeKR0} satisfies trivially the assumptions \eqref{AbeKR}. In particular, we have $$g_0=D_s\tilde{g}\big|_{s=0}=-\varepsilon.$$ We now derive stability conditions for the phase boundary \eqref{shock} under the Abe\-ya\-ratne-Knowles rule \eqref{AbeKR0}. This amounts to looking for zeroes of $\widehat{\Delta}$ when $g_0=-\varepsilon<0$. Since $\widehat{\Delta}$ is homogeneous of order three in $(\lambda,\xi)$, we have from \eqref{Lopaq2} that
\[
\widehat{\Delta}(\mu^{1/2}\lambda,\xi)=\xi^{3}\widehat{\Delta}(\mu^{1/2}\lambda/\xi,1)=\xi^{3}\widehat{\Delta}(z,1),
\]
where $z:=\mu^{1/2}\lambda/\xi$. Thus, for $\xi\neq0$ (see Corollary \ref{cor1dstab} for the case $\xi=0$), the Lopatinski\u{\i} determinant \eqref{Lopaq2} vanishes if and only if $\widehat{\Delta}(z,1)$ does. So, without loss of generality, we look for roots of
 \begin{equation}
 \label{Gfun1}
 G(z):=\widehat{\Delta}(z\sqrt{\mu},1)=-4(z^2+\delta_1)\sqrt{z^2+1}-\delta z(z^2+\delta_2),
 \end{equation}
where $\delta=9J_0\varepsilon\sqrt{\mu}>0$. Notice that in the limit when $\varepsilon\to0$ one recovers the Maxwell rule, under which the phase boundary \eqref{shock} is strongly unstable for $\delta_1<0$. Observe that, in this case ($\delta_1<0$), the phase boundary also remains strongly unstable under the Abeyaratne-Knowles rule, because $G$ necessarily admits at least a  positive real  zero for all $\varepsilon>0$. Indeed, $G$ is continuous and real valued along the positive real axis, with $G(0)=-4\delta_1>0$ and $$G(\sqrt{-\delta_1})=-\delta\sqrt{-\delta_1}(-\delta_1+\delta_2),$$ which is negative since from \eqref{Lopaq2}, we have $\delta_1<\delta_2<1$ for all $\mu>0$ and $J_0>1$. The intermediate value theorem ensures the existence of at least one root of $G$ on $(0,\sqrt{-\delta_1})$. This, in turn, implies that $\widehat{\Delta}$ vanishes in $\Gamma^+$.   

Therefore, for the analysis to come we assume $\delta_1\geq0$. In this case, we have that 
 \begin{equation}\label{indel1}
 0\leq\delta_1<\delta_2<1,
 \end{equation} 
 for all material parameters $\mu>0$ and $J_0\geq1$. 
 
 First,  we proceed with the investigation of possible zeroes of the Lopatinski\u{\i} determinant along the imaginary axis, which are associated to the existence of surface waves. 
Let us consider a zero of $G$ of the form $z = \ii t$, $t\in\R$. Let us define $\widetilde{G}(t):=G(\ii t)$ and we investigate whether $\widetilde{G}$ has real zeros. Let us first consider $t\in[-1,1]$. In this case we can write
 \begin{equation}
 \label{tildeG0}
 \widetilde{G}(t)=-4(\delta_1-t^2)\sqrt{1-t^2} - \ii \delta t (\delta_2-t^2).
 \end{equation}
 $\Re \widetilde{G}(t)$ vanishes at $t=\pm1$ or  $t=\pm\sqrt{\delta_1}$. Substituting into the imaginary part one obtains 
 \[
 \Im G(t)\big|_{t=\pm1}=\mp\delta(\delta_2-1)\neq 0, \quad \text{and} \quad \Im G(t)\big|_{t=\pm\sqrt{\delta_1}}=\mp\delta \sqrt{\delta_1}(\delta_2-\delta_1).
 \]
In view of \eqref{indel1} and $\delta>0$, we have that if $\delta_1>0$ then the imaginary part of $\widetilde{G}$ never vanishes when the real part does. So, we conclude that  $\widetilde{G}$ does not vanish on the interval $[-1,1]$ when $\delta_1>0$.  However if $\delta_1=0$, both the real and imaginary parts of $\widetilde{G}$ vanish simultaneously if and only if  $t=0$, implying that $t=0$ is the unique root of $\widetilde{G}$ in the interval $[-1,1]$. Now suppose that  $|t|>1$. In this case we have $\sqrt{-t^2+1}=\ii\,\sgn(t)\sqrt{t^2-1}$ and hence
 \begin{equation}
 \label{tildeG}
  \widetilde{G}(t)=\ii\Big[4\ (t^2-\delta_1)\,\sgn(t)\sqrt{t^2-1}+\delta t(t^2-\delta_2)\Big].
 \end{equation} 
 We immediately observe that since $\delta_1<1<t^2,\: \delta_2<1<t^2$ then both terms inside the brackets in \eqref{tildeG} are strictly positive if $t>1$ and strictly negative if $t<-1$. This shows that $\widetilde{G}(t)$ does not vanish for $|t|>1$. For the analysis to come, it is convenient to look for zeros of $G(z)$ along the positive real axis. Set $z=t>0$ and substitute into \eqref{Gfun1} to obtain
 \begin{equation}
 \label{Gpost}
G(t) =-\Big(4(t^2+\delta_1)\sqrt{t^2+1}+\delta t (t^2+\delta_2)\Big).
\end{equation}
It is clear that $G$ is strictly negative along the positive real axis, in view that $\delta>0$ and $\delta_1>0$. We summarize our last findings as follows.

 \begin{lemma}[existence of positive real and purely imaginary zeroes]
 \label{imzero1}
 If $\delta_1>0$ then the mapping $z \mapsto G(z)$ does not have positive real or purely imaginary zeroes. When $\delta_1=0$ the unique root of $G(z)$ on the imaginary axis is $z=0$ and there are not positive real zeroes.
 \end{lemma}
 The remaining task is to determine whether there exist zeroes of $G$ in the right complex half-plane, $\Re z>0$. Since $\delta_1,\delta_2,\delta$ are real, and the principal branch of the squared root satisfies
 \[
 \sqrt{\:\overline{z}\:}=\overline{\sqrt{z}},\quad z\notin\R^{-},
 \] 
then it easy to see that $G(\overline{z})=\overline{G(z)}$. Therefore, the zeroes of $G$ lie symmetrically with respect to the real axis as complex conjugate pairs. So, in order to count the zeroes of $G$ in the considered region, it is enough to count them on the first complex quadrant 
\[
\Ic_{+}:=\{z\in\C:\:\Re z>0, \; \Im z>0\}.
\]
We apply the argument principle to count the number of roots of $G$ in $\Ic_{+}$. We analyze the cases $\delta_1>0$ and $\delta_1=0$ separately.

\subsection*{The case $\delta_1>0$}
Consider $G(z)$ as $z$ varies counterclock-wise along the closed contour $\mathcal{C}_R$ consisting of a quarter of a circle of radius $R$ together with two straight segments--one horizontal and one vertical--connecting the end points of the arc to the center (the origin); see Figure \ref{figcontourC}. From Lemma \ref{imzero1}, it is known that if $\delta_1>0$ then $G$ has neither purely imaginary roots nor positive real roots,  and we only have to avoid the branch cut of the square root when we map the portion of the positive imaginary axis. We are interested in the behavior of the image of $\mathcal{C}_R$ under $G$ as $R\to\infty$. First, we examine the mapping of the  portion of $\mathcal{C}_R$ on the real axis, that is, we evaluate $G$ when $z=t$, $t>0$.  Equation \eqref{Gpost} shows that the function $t\to-G(t)$ is the sum of two strictly increasing real-valued functions on the positive real axis. So, $G$ strictly decreases from $G=-4\delta_1$ at $t=0$ to $G=-\infty$ as $t\to\infty$.  Thus, the image of the horizontal portion $[0,R]$, of $\mathcal{C}_R$ under $G$ is the interval $\big[G(R),-4\delta_1\big]$ along the negative real axis, with $G(R)\to -\infty$ as $R \to \infty$. On the other hand, from \eqref{tildeG0} and \eqref{tildeG} we have that $G$ along the positive imaginary axis becomes
\[
G(\ii t)=\widetilde{G}(t)= - \ii \delta t(\delta_2-t^2) -4(\delta_1-t^2) \begin{cases}  \ii \sqrt{t^2-1}, & t > 1, \\ \sqrt{1-t^2}, & 0\leq t \leq 1. \end{cases}.
\]

\begin{figure}[t]
\begin{center}
\subfigure[Contour $\mathcal{C}_R$]{\label{figcontourC}\includegraphics[scale=.7, clip=true]{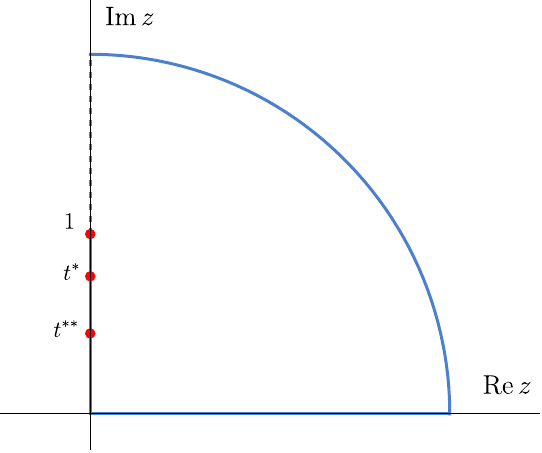}}\quad\qquad
\subfigure[$G(\mathcal{C}_R)$]{\label{figcontourGC}\includegraphics[scale=.9, clip=true]{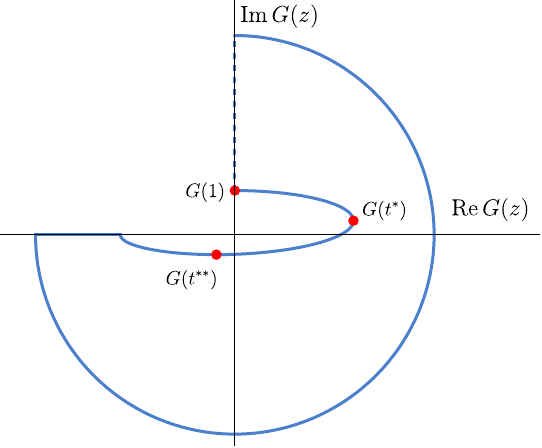}}
\end{center}
\caption{Illustration of the contour $\mathcal{C}_R$ in the $z$-complex plane (in blue; panel (a)) and of its image under the mapping $G$ (panel (b); color online).}\label{figcontour}
\end{figure}

Hence, there are two cases to consider when mapping the positive imaginary axis. If $t>1$, notice that $G$ takes purely imaginary values. That is 
\[
G(\ii t)=\widetilde{G}(t)=\ii q(t),
\] 
where $q$ is a real-valued function that, by virtue of the previous argument, strictly increases for $t>1$ starting from the value  $(1-\delta_2)\delta$ at $t=1$ up to $\infty$ as $t$ approaches infinity. Thus, $G$ maps the segment $[\ii,\ii R]$ ($R>1$) into the interval $\big[\ii(1-\delta_2)\delta,\ii q(R)\big)$. 
On the other hand, if $0<t\leq1$ then it is not hard to verify that the function 
\[
t \mapsto \Re(G(\ii t))=4(t^2-\delta_1)\sqrt{1-t^2},
\]
which increases from $-4\delta_1$ (at $t=0$) up to a root at $t=\sqrt{\delta_1}$, reaches its unique maximum  at $t = t^*:=\sqrt{(2+\delta_1)/3}\in(0,1)$ and finally vanishes again at $t=1$. Meanwhile,  the imaginary part
\[
t\to \Im(G(\ii t))=\delta t(t^2-\delta_2),
\] 
is a convex real function of  $t\in [0,1]$ with a negative minimum attained at $t=t^{**}:=\sqrt{\delta_2/3}\in(0,1)$ and roots at $t=\sqrt{\delta_2}>t^{**}$ and $t=0$. Consequently, $G$ maps the segment $\ii [0,1]$ along the imaginary axis into a pseudo-circular arc surrounding the origin, with end points $G(0)=-4\delta_1$, $G(1)=\delta(1-\delta_2)\ii$, and  extreme points $G(t^*)$,  $G(t^{**})$; see Figure \ref{figcontourGC}. Finally, we examine the mapping of the circular portion of $\mathcal{C}_R$. To that end, we proceed as in \cite{JL}, introducing polar coordinates $(R,\phi)$ and defining the function
\[
H(R,\phi):=G(w),\quad w=Re^{\ii \phi}.
\]
From expression \eqref{Gfun1}, we notice that the image of the circular portion for large $R$ behaves like
\[
H(R,\phi) \approx (-4-\delta)w^3=-(4+\delta) R^3 e^{\ii 3\phi},
\]
as $R \to \infty$. Hence, the image is almost a circular portion (three-quarters) too, with negative orientation.

From the preceding discussion, we conclude that the origin $w=0$ in the codomain lies outside of the image of the contour $G(\mathcal{C}_R)$, as illustrated in Figure \ref{figcontourGC}. This implies that there is no change in the argument of $G(z)$ as $z$ varies counter-clock wise around the closed contour $\mathcal{C}_R$ and that there are no roots of $G$ inside the contour $\mathcal{C}_R$. The argument can be applied to any arbitrary large radius $R>0$. Therefore, as long as $\delta_1>0$, $G$ does not vanish in $\Ic_{+}$, and then  $G$ has no roots with positive real part. This also implies $\widehat{\Delta}(\lambda,\xi)$ does not vanish for $\Re\lambda>0$. We conclude that, under the Abeyaratne-Knowles kinetic rule, $\delta_1>0$ is a sufficient condition for uniform (or strong) stability for the phase boundary $(\bFu,\bm{0},0,\eu)$.

\subsection*{The case $\delta_1=0$}

When $\delta_1=0$, the Lopatinski\u{\i} determinant takes the form $G(z)=-z^2G_1(z)$, where $G_1$ is given by
\[
G_1(z)=4\sqrt{z^2+1}+\delta\Big(z+\frac{\delta_2}{z}\Big).
\]
$G_1$ does not have zeroes in $\{\Re z>0\}$, since the real part of $G_1$ does not vanish whenever $\Re z>0$. Indeed, standard algebraic manipulations yield
\[
\Re \big( G_1(z)\big) = 4\Re\big(\sqrt{z^2+1}\big)+\delta(\Re z)\Big(1+\frac{\delta_2}{|z|^2}\Big),
\]
which is strictly positive for $\Re z>0$, provided that the square root in the first term is assumed to be the principal branch and then $\Re\big(\sqrt{z^2+1}\big)>0$ whenever $\Re z\neq0$. The unique root of $G$ in for $\Re z\geq0$ is $z=0$, which corresponds to roots of $\widehat{\Delta}$ of the form $(0,\pm1)\in\Gamma$. Thus, the phase boundary is weakly stable. Recall that from Theorem \ref{MaxSta}, $\delta_1=0$ for each $J_0>1$ if and only if $\mu=\mu_*$. 

We summarize our findings into the following main result.
\begin{theo}[stability criteria]\label{AbeySta}
Consider a compressible hyperelastic Hadamard material determined by the volumetric function \eqref{hdef1} with material parameters satisfying $\mu>0$, $J_0>1$  and the phase boundary $(\bFu,\bm{0},0,\eu)$.  Assuming the Abeyaratne-Knowles rule \eqref{AbeKR0}, we have the following stability criteria:

\begin{itemize}
\item [\rm{(i)}] If $J_0>1$ and $\mu>\mu_*$, which amounts to $\delta_1>0$,  then the phase boundary  is uniformly stable for all $\varepsilon>0$.
\item [\rm{(ii)}] If $J_0>1$ and $\mu=\mu_*$, which amounts to $\delta_1=0$, then the phase boundary is weakly stable for all $\varepsilon>0$.
\item [\rm{(iii)}] If $J_0>1$ and $0<\mu<\mu_*$, which amounts to $\delta_1<0$, then the phase boundary is strongly unstable for all $\varepsilon>0$.
 \end{itemize}
 \end{theo}

\section{Discussion}
\label{secdisc}

 A concise summary of the stability results established in Theorems \ref{MaxSta} and \ref{AbeySta} can be stated as follows: for large enough shear modulus $\mu>0$, the phase boundaries $(\bFu,\bm{0},0,\eu)$ exhibit dynamical stability, being weakly stable under the Maxwell criterion and uniformly stable under regular perturbations of Abeyaratne-Knowles type. In the variational context, this is consistent with a recent result which states that, for Hadamard materials with large enough shear modulus, $\mu>0$, the entire jump set becomes stable (see \cite{GrTr19}). Since $(\bFu,\bm{0},0,\eu)$ belongs to the jump set for all $\mu>0$ and $J_0>1$ (see Remark \ref{Eder1}), it is stable for large $\mu$. More precisely, given $J_0>1$, the smallest value of $\mu$ for which (variational) stability of $(\bFu,\bm{0},0,\eu)$ can be guaranteed via the generalized Legendre-Hadamard conditions \cite{GrTr16} is given by 
 \[
 \mu=\mu_{**}:=\frac{2}{3}\Big(1+\frac{1}{\sqrt{3}J_0}\Big);
 \]
 see Appendix \ref{Exa1}. It not hard to verify from \eqref{mu:barr} that
 \[
 \mu_{**}>\mu_*,\quad \text{for all}\; J_0>1.
 \]
 Therefore, Theorems \ref{MaxSta} and \ref{AbeySta} imply that for $\mu>\mu_{**}$, the phase boundaries $(\bFu,\bm{0},0,\eu)$ are not only stable in the variational sense but also dynamically stable. Conversely, in the regime of small $\mu>0$,  dynamical stability is lost, as established in Theorems \ref{MaxSta} and \ref{AbeySta}. This behavior is consistent with its counterpart in the Calculus of Variations. Indeed, in a very recent paper, Grabovsky and Truskinovsky \cite{GrTr24} analyzed the limit $\mu\to0^+$ (liquid limit) and showed that certain parts  of the jump set become unstable. The simultaneous loss of stability (both variational and dynamical) may be explained by the fact that small perturbations of $\bF$ can produce large changes in $J = \det\bF$, allowing for variations that reduce the energy due to the non-convexity of $h(J)$, while the increment in $\mu|\bF|^2$ remains comparatively small. 
  
While a formal connection between variational and dynamical stability has not yet been established, this work shows that for a class of phase boundaries in Hadamard materials, both notions coincide within a specific range of material parameters. A forthcoming paper will aim to determine bounds for dynamical stability in Hadamard materials and compare them with those  previously established for variational stability by Grabovsky and Truskinovsky \cite{GrTr19}. On the other hand, it is obvious that the direct connection between the elliptic Lopatinski\u{\i} condition \cite{Lopa56} (also known as the Lopatinski\u{\i}-Shapiro condition \cite{DHP03,Shpr53} or Complementing Boundary Condition \cite{ADN-II}) and the hyperbolic Kreiss-Majda-Lopatinski\u{\i} condition used here establishes an intimate link between dynamic and static stability of elastic interfaces. This will be detailed in the forthcoming work \cite{FMPV2}.

\section*{Acknowledgements}

\section*{Summary Statement}
\subsection*{Funding declarations}
The work of L. Morales was supported by the Secretar\'{\i}a de Ciencia, Humanidades, Tecnolog\'{\i}a e Innovaci\'{o}n (SECIHTI -- Ministry of Science, Humanities, Technology and Innovation), Mexico, through the Program ``Estancias Post-doctorales por M\'exico 2022''. The work of R. G. Plaza and F. Vallejo was fully supported by SECIHTI, Mexico, grant CF-2023-G-122. 

\subsection*{Conflict of interest} The authors declare no conflict of interest.
\subsection*{Author contributions} The authors contributed to and reviewed all article sections equally.
\subsection*{Ethics declaration} Not applicable.
\subsection*{Data Availability} No datasets were generated or analyzed during the current study.

\appendix

\section{Generalized Legendre-Hadamard condition for phase boundaries}
\label{GrabTro}

For the convenience of the reader, we summarize the results by Grabovsky and Truskinovsky \cite{GrTr16}  that are relevant to this work. Within the framework of nonlinear elasticity, equilibrium configurations  of an elastic solid under prescribed loads and suitable boundary conditions are usually identified with minimizers of the elastic energy functional 
\begin{equation}
\label{elasen1}
\mathcal{E}(\bm{y}) = \int_\Omega W(\nabla \bm{y}(\bm{x}))\, d\bm{x} 
       \;-\; \int_{\partial\Omega} t(\bm{x})\cdot \bm{y}\, dS(\bm{x}),
\end{equation}
where, as before, $\Omega\in\R^d$ is an open bounded domain that represents the reference configuration of the elastic solid, $\bm{y}:\Omega\to\R^{d}$ is the displacement vector of the deformation and $W$, of class $C^2$, is the stored energy density function defined on the space of square matrices with positive determinant (see \cite{GiaHil96vI}). In this context, the Legendre-Hadamard condition plays an essential role, as it constitutes a necessary condition for a smooth single deformation to be a weak local minimizer of $\mathcal{E}$ (cf. Dacorogna \cite[Ch. 5]{Dac08} and Grabovsky \emph{et al.} \cite{GKT11}). However, when the stored energy density function $W$ that characterizes a material is non-convex,  the energy \eqref{elasen1} can be minimized by mixing several deformation gradients that coexist along smooth surfaces. The resulting microstructure often organizes into patterns called laminates and correspond to strong local minima of \eqref{elasen1}. The simplest case are  phase boundaries of the form \eqref{eq:twin} (solid-solid phase transformation) in which the deformation gradient takes two values at both sides of the planar interface with normal $\bn$:
\begin{equation}\label{eq:twin1}
\nabla \bm{y} =
\begin{cases}
\bF^{-}, & \boldsymbol{x}\cdot\bn<0,\\[4pt]
\bF^{+}, & \boldsymbol{x}\cdot\bn>0.
\end{cases}
\end{equation}
A question of fundamental importance is which  properties  the pair $\bFs$ must satisfy in order to be associated with strong local minimizer of \eqref{elasen1} via \eqref{eq:twin1}.  The classical Legendre-Hadamard condition fails to yield meaningful information to characterize any given non-regular deformation as a strong local minimum. A suitable algebraic generalization of the classical Legendre-Hadamard condition for  phase boundaries of the form \eqref{eq:twin1} was developed in \cite{GrTr16}, resulting in necessary conditions for \eqref{eq:twin1} to be a strong local minimizer. They are referred to as generalized Legendre-Hadamard conditions for phase boundaries. We briefly outline  this approach  and  introduce key definitions for clarity (see also \cite{GrTr19b}).

\begin{definition}[the binodal]\label{Bin1}
We say that $W$ is \emph{quasiconvex at $\bF \in \mathbb{R}^{d \times d}$} if
\begin{equation}\label{sstab1}
\int_{D} W\bigl(\bF + \nabla \varphi(\bm{x})\bigr)\,d\bm{x} \;\ge\; |D|\, W(\bF),
\end{equation}
for any domain $D \subset \mathbb{R}^d$ and any $\varphi \in C^{\infty}_0(D;\mathbb{R}^d)$. The set of points $\bF$ where the quasiconvexity fails is called the \emph{binodal region} for $W$ and is denoted by $\mathcal{B}$. The boundary of $\mathcal{B}$, $\partial\mathcal{B}$, is called \emph{the binodal}.
\end{definition}

Quasiconvexity is referred to as strong local stability \cite{GrTr14} (in a variational sense). Indeed,  if $W$ is quasiconvex at $\bF$ then \eqref{sstab1} guarantees that the energy cannot be lowered by any admissible (strong) variation of $\bF$. In turn, as a stability threshold, the binodal characterizes strongly marginally stable states \cite{GrTr13} and plays a central role in distinguishing local minima from saddle points in vectorial variational problems \cite{GrTr19b}. Nevertheless, an explicit characterization of the binodal is difficult in practice because, unlike the Legendre-Hadamard condition, the quasiconvexity condition lacks a pure algebraic characterization. To circumvent this difficulty, the authors in \cite{GrTr16} consider a set of points associated to a simpler algebraic property known as the \emph{Weierstrass convexity condition} which is a consequence of quasiconvexity.  $W$ satisfies the Weierstrass convexity condition at $\bF$ if
\begin{equation}\label{Weconv}
\cEw(\bm{v},\bm{w};\bF):=W(\bF+\bm{v}\otimes\bm{w}) - W(\bF) - \bm{\sigma}(\bF): ( \bm{v}\otimes\bm{w})\geq0,
\quad
\forall \bm{v},\,\bm{w}\in\R^d,
\end{equation}
where $\cEw$ denotes the Weierstrass excess function.  Grabovsky and Truskinovsky \cite{GrTr16} study the boundary points of the set of deformation gradients $\bF$ satisfying \eqref{Weconv}. At a point where $\bF$ crosses this boundary, necessarily there is a pair $(\bm{a},\bn)$ such that $\cEw(\bm{a},\bn;\bF)=0$.  Non trivial zeros of $\cEw$ with $\bm{a}\neq0$ and $|\bn|=1$ are of fundamental importance, as they describe the set of homogeneous deformation gradients $\bF=\bF^-$ that permit energy-neutral nucleation of layers of a new phase with a compatible deformation gradient $\bF^- + \bm{a} \otimes \bn$ \cite{GrTr19,GrTr16}. Observe that when $\bF^{-}$ crosses the boundary of validity of \eqref{Weconv}, while $\cEw(\bm{a},\bn,\bF^{-})=0$, the function $\cEw(\bm{v},\bm{w};\bF^{-})$ is still non-negative and, hence, it attains a minimum at $(\bm{a},\bn)$. Since $W$ is of class $C^2$, it follows that
\[
\nabla_{(\bm{v},\bm{w})}(\bm{a},\bn,\bF^{-})=\Big(\llb \bm{\sigma}(\bF) \rrb\bn,\llb \bm{\sigma}(\bF) \rrb^{\top}\bm{a}\Big)=(\bm{0},\bm{0}).
\]
The vanishing of $\cEw$ and the equations above define the so called \emph{jump set} (cf. \cite{GrTr16,GrTr19b}).

\begin{definition}\label{jumSt}
The jump set $\cJ$ is the closure of the set of gradients $\bFs$ for which there exist $\bm{a}\in\R^{d}\setminus\{0\}$ and a unit vector $\bn\in\R^d$, such that
\begin{equation}
\label{Jump3}
\begin{aligned}
&\llb \bF \rrb = \bm{a} \otimes \bn, \\
&\llb W \rrb - \bm{\sigma}(\bF^{-}): \llb \bF \rrb= 0,\\
&\llb \bm{\sigma}(\bF) \rrb\bn=0,\\
\text{and} \quad &\llb \bm{\sigma}(\bF) \rrb^{\top}\bm{a}=0.
\end{aligned}
\end{equation}
\end{definition}

The first three equations in \eqref{Jump3} constitute classical conditions widely used in the modeling of phase coexistence across a smooth surface \cite{GrTr11}. In general, a necessary condition for $\bFs$ to be associated with a strong local minimizer of \eqref{elasen1} via \eqref{eq:twin1} is that $\bFs$ belongs to the subset of the jump set that lies within the binodal \cite{GrTr19b}. 

Given that \eqref{Weconv} is consequence of quasiconvexity \cite{GrTr16}, and that $\mathcal J$ lies  on the boundary of validity of \eqref{Weconv}, it follows that every neighborhood of each point in $\mathcal J$ contains points where \eqref{Weconv} (and therefore quasiconvexity) fails to hold. Hence, $\cJ$  necessarily lies within the closure of the binodal region ($\mathcal{B}\cup\partial\mathcal{B}$) \cite{GrTr14,GrTr19,GrTr11}. In this context, Grabovsky and Truskinovsky  \cite{GrTr16} aim to determine those points of the jump set $\cJ$ that lie entirely on the binodal by testing quasiconvexity; if quasiconvexity holds there, such points necessarily lie in the binodal (see \cite{GrTr19} for details). In other words, the goal is to identify a subset of the binodal as  a  stable portion of the jump set. To that end, Grabovsky and Truskinovsky  \cite{GrTr16} perform second order expansions of $W$ around $\bFs\in\cJ$ and exploit the fact that both gradients must lie on the binodal. As a result,  generalized Legendre-Hadamard conditions are obtained in the form of inequalities describing the portion of the jump set $\cJ$ that lies in the binodal. 

\section{Generalized Legendre-Hadamard conditions for Hadamard materials}
\label{L-Hgt}

Although the generalized Legendre-Hadamard conditions are quite general and far from trivial,  Grabovsky and Truskinovsky  showed that Hadamard materials of the form \eqref{Hadamardmat0}, with $h$ having the double-well shape (see \cite{GrTr16} for a precise definition), constitutes a wide class of materials for which both the jump set $\cJ$ and these conditions admit explicit expressions. As a result, the corresponding jump set and the generalized Legendre-Hadamard conditions can be formulated in terms of $h$ and the eigenvalues $\theta^{\pm}_1,\cdots,\theta^{\pm}_d$ of the matrix (see \cite{GrTr16} for details)
\[
\widehat{\bm{G}}_{\pm}:=\Cof \big((\bFs)^{\top} \bFs\big).
\]
In what follows, we summarize the main results from \cite{GrTr16}, specialized to the case of Hadamard materials. Equations one, three, and four in \eqref{Jump3} imply that $\bn$ must be a common eigenvector of both $\widehat{\bm{G}}_{-}$ and $\widehat{\bm{G}}_{+}$ associated with the same eigenvalue $\theta^{+}_1=\theta^{-}_1=\theta$. Moreover, these equations also yield the algebraic relation
\begin{equation}\label{preJu1}
\frac{\llb h'\rrb}{\llb J\rrb}=-\frac{\mu}{\theta}.
\end{equation}
On the other hand, the second equation in \eqref{Jump3} can be written as  
 \begin{equation}\label{preJu2}
\llb h\rrb-\langle h'\rangle\llb J\rrb=0,
\end{equation}
which is known as the equal area rule. Due to the double well-shape assumption on $h$, \eqref{preJu2} defines a one to one correspondence between $J^+$ and $J^{-}$. So the jump set reduces to equations \eqref{preJu1}-\eqref{preJu2}.  Taking into account that $J^-$ can be obtained in terms of $\theta^{-}_1,\cdots,\theta^{-}_d$ (see Eq. \eqref{cofdet1}), then \eqref{preJu1}-\eqref{preJu2} represent the jump set $\cJ$ as the surface in the $\theta$-space of eigenvalues of $\widehat{\bm{G}}_{-}$. It can be proved that $\widehat{\bm{G}}_{+}$ and $\widehat{\bm{G}}_{-}$ admit the same eigenvectors, including $\bn$ corresponding to their common eigenvalue $\theta$. The remaining eigenvalues $\theta_{k}^{\pm}$, $k=2,\cdots, d$, obey the relation  
\[
\frac{\theta_{k}^{+}}{(J^+)^2}=\frac{\theta_{k}^{-}}{(J^-)^2},\quad k\geq2.
\]
 
Once the structure of the jump set is established, the generalized Legendre-Hadamard conditions characterize the stable portion of the jump set. They can be written in terms of $h$ and the eigenvalues of $\widehat{\bm{G}}_{\pm}$ as follows (see \cite{GrTr16} for details):

\begin{enumerate}
\item \textit{Stability with one phase fixed}: If $J^+>J^{-}$ and $\theta'_{\pm}$ is the second largest eigenvalue of $\hat{G}_{\pm}$ then it must be satisfied that
\begin{equation}\label{S1}
\tag{S$_1$} 
\theta'_{+}\leq\theta, \qquad \theta'_{-}\leq\theta\left(\frac{J^-}{J^+}\right)^2.
\end{equation}
\item \textit{Stability of laminates}:
\begin{equation}\label{S2}
\tag{S$_2$} 
\frac{\mu}{\theta'_{\pm}}\geq\frac{\bigl(h''(J^{\pm}) + \mu\theta^{-1} (J^{\mp}/J^{\pm})\bigr)^2}{\mu\theta^{-1} + h''(J^\pm)}
-  h''(J^{\pm}).
\end{equation}
\end{enumerate}

Conditions \eqref{S1} and \eqref{S2}  determine the portion of the jump set that lies in the binodal. In the paper \cite{GrTr19}, the authors prove that for $\mu$ large enough, the jump set becomes stable, which implies that this set delivers the entire binodal.  

\begin{remark}\label{ReduLG}
Note that if $h''(J^{\pm})=0$, \eqref{S2} reduces to
\[
\theta'_{+}\leq\theta\left(\frac{J^+}{J^-}\right)^2, \qquad \theta'_{-}\leq\theta\left(\frac{J^-}{J^+}\right)^2,
\]
which is clearly a consequence of  \eqref{S1}, because $J^{+}>J^{-}$ implies the first equation above. Indeed, the first equation in \eqref{S1} gives
\[
\theta'_{+}\leq\theta<\theta\left(\frac{J^+}{J^-}\right)^2.
\]
Therefore when $h''(J^{\pm})=0$, the generalized Legendre-Hadamard conditions \eqref{S1}, \eqref{S2} reduce to \eqref{S1}. 
\end{remark}

\begin{remark}
Grabovsky and Truskinovsky \cite{GrTr16} actually derived explicit expressions for the jump set and for the generalized Legendre-Hadamard conditions in the special case of Hadamard materials \eqref{Hadamardmat0} with $\mu=1$ (see \cite{GrTr16}, Eqs. (5.6)-(5.7) and (5.10)-(5.11)). However, it is to be observed that the properties defining both the jump set and the generalized Legendre-Hadamard conditions, \eqref{Weconv} and \eqref{sstab1}, respectively, are invariant under multiplications by positive scalars. Therefore,  both the jump set and the stability conditions for $W$ given in \eqref{Hadamardmat0} with arbitrary $\mu$ can be obtained from those corresponding to $W$ with $\mu=1$, through the rescaling $h\to (1/\mu)h$.
\end{remark}

\subsection{An example} 
\label{Exa1}
Let us verify that, under the assumption
\begin{equation}\label{cons3}
\mu>\frac{2}{3}\Big(1+\frac{1}{\sqrt{3}J_0}\Big),
\end{equation}
the phase boundaries $(\bFu,\bm{0},0,\eu)$ lie in the stable portion of jump set for Hadamard materials \eqref{Hadamardmat0} with $h$ given by \eqref{hdef1}. Notice that the second inequality in \eqref{cons3} is trivially satisfied for large values of $\mu$.  We have to verify equations \eqref{preJu1}, \eqref{preJu2} of the jump set as well as conditions \eqref{S1} and \eqref{S2} to evaluate stability. Let us first compute the eigenvalues of $\widehat{\bm{G}}_{\pm}$. Straightforward computations yield
\[
\begin{aligned}
\widehat{\bm{G}}_{\pm}=\Cof \big((\bFu)^{\top} \bFu\big)
=\begin{pmatrix}
 \theta_0 & 0\vspace{.3cm}\\
0 & \big(J_0\pm\frac{1}{\sqrt{3}}\big)^2\frac{1}{\theta_0}
\end{pmatrix}.
\end{aligned}
\]
The eigenvalues of $\widehat{\bm{G}}_{\pm}$ are $\theta_{1}^{\pm}=\theta_0$ associated to the common eigenvector $\bn=\eu$ and $\theta_{2}^{\pm}=\big(J_0\pm\frac{1}{\sqrt{3}}\big)^2\frac{1}{\theta_0}$. It is easy to see from \eqref{cons3} that $3\mu J_0/2>J_0+(1/\sqrt{3})$. Therefore
\begin{equation}\label{ineq5}
\theta_0^2=\Big(\frac{3\mu J_0}{2}\Big)^2>\Big(J_0+\frac{1}{\sqrt{3}}\Big)^2>\Big(J_0-\frac{1}{\sqrt{3}}\Big)^2>0,
\end{equation}
in view that $J_0>1$. \eqref{ineq5}  implies that $\theta=\theta_0$ is the largest eigenvalue of both $\widehat{\bm{G}}_{\pm}$ and $\theta'_{\pm}=\theta_{2}^{\pm}$ are the second largest eigenvalues, respectively.  Upon direct substitution of \eqref{hdef1}, Eq. \eqref{preJu2} becomes 
\[
J^+=2J_0-J^{-},
\]
which gives the one to one correspondence between $J^{\pm}$ and is trivially satisfied by $\bFu$ (see \eqref{detpos}). On the other hand, direct substitution of \eqref{hdef1}, \eqref{teta0}, \eqref{detpos} and \eqref{ache0} into \eqref{preJu1} yields
\[
\frac{\llb h'\rrb}{\llb J\rrb}=\frac{2h_0}{2/\sqrt{3}}=\frac{-2\mu}{3J_0\mu}=-\frac{\mu}{\theta_0}.
\]
That is, $(\bFu,\bm{0},0,\eu)$ also satisfies \eqref{preJu1}, and then it lies in the jump set. Now, let us check that   the generalized conditions of Legendre-Hadamard type, \eqref{S1} and \eqref{S2}, are also satisfied. Indeed, since $h''(J^{\pm})=0$, it follows from Remark \ref{ReduLG} that both \eqref{S1} and  \eqref{S2} reduce to  \eqref{S1}; hence, it suffices to verify property \eqref{S1} only. Direct computations yield
\[
\frac{\theta'_{+}}{\theta}=\frac{\theta'_{-}(J^{+})^2}{\theta(J^{-})^2}=\left(\dfrac{J_0+\tfrac{1}{\sqrt{3}}}{\theta_0}\right)^2\leq1,
\]
inasmuch as \eqref{ineq5} holds. Thus, \eqref{S1} and \eqref{S2} are satisfied and, consequently, $\bFu$ lies on the portion of the jump set that is contained in the binodal. In summary, for $h$ given by  \eqref{hdef1}, the phase boundary $(\bFu,\bm{0},0,\eu)$ lies within the stable part of the jump set whenever \eqref{cons3} holds.

\begin{remark}
Note that for $\mu=1$, the second assumption in \eqref{cons3} becomes $J_0>2/\sqrt{3}$. This is a fundamental assumption for the stability conditions \eqref{S1} and \eqref{S2} to capture a stable portion of the jump set in the case $\mu=1$ and $h$ given by  \eqref{hdef1} (see \cite{GrTr16}, \S 5.5).
\end{remark}

%
%
%
%
%

\def\cprime{$'\!\!$}




%

\end{document}